\documentclass[reqno]{amsart}
\usepackage{amssymb,amsthm,amsfonts,amstext}
\usepackage[mathscr]{euscript}
\usepackage{amsmath}

\makeatletter \@addtoreset{equation}{section} \makeatother

\renewcommand\thefigure{\thesection.\@arabic\c@figure}
\renewcommand\thetable{\thesection.\@arabic\c@table}
\newtheorem{theorem}{Theorem}[section]
\newtheorem{lemma}[theorem]{Lemma}
\newtheorem{proposition}[theorem]{Proposition}
\newtheorem{corollary}[theorem]{Corollary}

\newtheorem{remark}[theorem]{Remark}

\newcommand{\mc}[1]{{\mathcal #1}}
\newcommand{\ms}[1]{{\mathscr #1}}
\newcommand{\mf}[1]{{\mathfrak #1}}
\newcommand{\mb}[1]{{\mathbf #1}}
\newcommand{\bb}[1]{{\mathbb #1}}
\newcommand{\bs}[1]{{\boldsymbol #1}}

\newcommand{\<}{\langle}
\renewcommand{\>}{\rangle}

\renewcommand{\Cap}{{\rm cap}}

\title{A Martingale approach to metastability}

\author{J. Beltr\'an, C. Landim}

\address{\noindent IMCA, Calle los Bi\'ologos 245, Urb. San C\'esar
  Primera Etapa, Lima 12, Per\'u and PUCP, Av. Universitaria cdra. 18,
  San Miguel, Ap. 1761, Lima 100, Per\'u. 
\newline e-mail: \rm
  \texttt{johel.beltran@pucp.edu.pe} }

\address{\noindent IMPA, Estrada Dona Castorina 110, CEP 22460 Rio de
  Janeiro, Brasil and CNRS UMR 6085, Universit\'e de Rouen, Avenue de
  l'Universit\'e, BP.12, Technop\^ole du Madril\-let, F76801
  Saint-\'Etienne-du-Rouvray, France.  \newline e-mail: \rm
  \texttt{landim@impa.br} }

\begin{document}


\begin{abstract}
  We presented in \cite{bl2,bl7} an approach to derive the metastable
  behavior of continuous-time Markov chains. We assumed in these
  articles that the Markov chains visit points in the time scale in
  which it jumps among the metastable sets. We replace this condition
  here by assumtpions on the mixing times and on the relaxation times
  of the chains reflected at the boundary of the metastable sets.
\end{abstract}

\maketitle

\section{Introduction}
\label{sec-1}


Cassandro et al. proposed in a seminal paper \cite{cgov} a general
method to derive the metastable behavior of continuous-time Markov
chains with exponentially small jump rates, called the pathwise
approach. In many different contexts these ideas permitted to prove
that the exit time from a metastable set has an asymptotic exponential
law; to provide estimates for the expectations of the exit times; to
describe the typical escape trajectory from a metastable set; to
compute the distribution of the exit (saddle) points from a metastable
set; and to prove the convergence of the finite-dimensional
distributions of the order parameter, the macroscopic variable which
characterizes the state of the process, to the finite-dimensional
distributions of a finite-state Markov chain. This approach has known
a great success, and it is impossible to review here the main results.
We refer to \cite{ov1} for a recent account of this theory.

In \cite{begk1, begk2}, Bovier et al. proposed a new approach to prove
the metastable behavior of continuous-time Markov chains, known as the
potential theoretic approach. Motivated by the dynamics of mean field
spin systems, the authors created tools, based on the potential theory
of reversible Markov processes, to compute the expectation of the exit
time from a metastable set and to prove that these exit times are
asymptotically exponential. They also expressed the expectation of the
exit time from a metastable set and the jump probabilities among the
metastable sets in terms of eigenvalues and right-eigenvectors of the
generator of the Markov chain.

Compared to the pathwise approach, the potential theoretic approach
does not attempt to describe the typical exit path from a metastable
set, but provides precise asymptotic formulas for the expectation of
the exit time from a metastable set. This accuracy, not reached by the
pathwise approach, whose estimates admit exponential errors in the
parameter, permits to encompass in the theory dynamics which present
logarithmic energy or entropy barriers such as \cite{cmt, bl3,
  clmst}. Moreover, in the case of a transition from a metastable set
to a stable set, it characterizes the asymptotic dynamics: the process
remains at the metastable set an exponential time whose mean has been
estimated sharply and then it jumps to the stable set.

As the pathwise approach, the potential theoretic approach has been
succesfully applied to a great number of models. We refer to the
recently published paper \cite{bbi1} for references.

Inspired by the evolution of sticky zero-range processes
\cite{bl3,l1}, dynamics which have a finite number of stable sets with
logarithmic energy barriers, we proposed in \cite{bl2, bl7} a third
approach to metastability, now called the martingale approach. This
method was succesfully applied to derive the asymptotic behavior of
the condensate in sticky zero-range processes \cite{bl3,l1}, to prove
that in the ergodic time scale random walks among random traps
\cite{jlt1,jlt2} converge to $K$-processes, and to show that the
evolution among the ground states of the Kawasaki dynamics for the two
dimensional Ising lattice gas \cite{bl5, gl5} on a large torus
converges to a Brownian motion as the temperature vanishes.

To depict the asymptotic dynamics of the order parameter, one has to
compute the expectation of the holding times of each metastable set
and the jump probabilities amid the mestastable sets.  The potential
theoretic approach permits to compute the expectations of the holding
times and yields a formula for the jump probabilities in terms of
eigenvectors of the generator. This latter formula, although
interesting from the theoretical point of view, since it establishes a
link between the spectral properties of the generator and the
metastable behavior of the process, is of little pratical use because
one is usually unable to compute the eigenvectors of the generator.

The martingale approach replaces the formula of the jump probabilities
written through eigenvectors of the generator by one, \cite[Remark 2.9
and Lemma 6.8]{bl2}, expressed only in terms of the capacities,
capacities which can be estimated using the Dirichlet and the Thomson
variational principles. We have, therefore, a precise description of
the asymptotic dynamics of the order parameter: a sharp estimate of
the holding times at each metastable set from the potential
theoretical approach, and an explicit expression for the jump
probabilities among the metastable sets from the aforementioned
formula.

This informal description of the asymptotic dynamics of the order
parameter among the metastable sets has been converted in \cite{bl2,
  bl7} into a theorem which asserts that the order parameter converges
to a Markov chain in a topology introduced in \cite{jlt2}, weaker than
the Skorohod one. The proof of this result relies on three hypotheses,
formulated in terms of the stationary measure and of the capacities
between sets, and it uses the martingale characterization of a
Markovian dynamics and the notion of the trace of a Markov process on
a subset of the configuration space.

In the martingale approach, the potential theory tools developped by
Bovier et al. \cite{{begk1, begk2}} to prove the metastability of
Markov chains can be very useful in some models \cite{bl3,
  l1} or not needed at all, as in \cite{jlt1, jlt2}. In these latter
dynamics, the asymptotic jump probabilities among the metastable sets,
which, as we said, can be expressed through capacities, are estimated
by other means without reference to potential theory.

The proof of the convergence of the order parameter to a Markov chain
presented in \cite{bl2, bl7} requires that in each metastable set the
time it takes for the process to visit a representative configuration
of the metastable set is small compared to the time the process stays
in the metastable set. We introduced in \cite{bl2} a condition,
expressed in terms of capacities, which guarantees that a
representative point of the metastable set is visited before the
process reaches another metastable set. This quite strong assumption,
fulfilled by a large class of dynamics, fails in some cases, as in
polymer models in the depinned phase \cite{cmt, clmst} or in the dog
graph \cite{sc1}. The main goal of this article is to weaken this
assumption.

More recently, Bianchi and Gaudilli\`ere \cite{bg11} proposed still
another approach based on the fact that the exit time from a set
starting from the quasi-stationary measure associated to this set is
an exponential random variable. The proof that the exit time from a
metastable set is asymptotically exponential is thus reduced to the
proof that the state of the process gets close to the quasi-stationary
state before the process leaves the metastable set. To derive this
property the authors obtained estimates on the mixing time towards the
quasi-stationary state and on the asymptotic exit distribution with
errors expressed in terms of the ratio between the spectral radius of
the generator of the process killed when it leaves the metastable set
and the spectral gap of the process reflected at the boundary of the
metastable set, a ratio which has to be small if a metastable behavior
is expected. They also introduced $(\kappa,\lambda)$-capacities, an
object which plays an important role in this article.  \smallskip

After these historical remarks, we present the main results of this
article. Consider a sequence of continuous-time Markov chains
$\eta^N(t)$. To describe the asymptotic evolution of the dynamics
among the metastable sets, let $X^N_t$ be the functional of the
process which indicates the current metastable set visited:
\begin{equation*}
X^N_t \;=\; \sum_{x=1}^\kappa x\, \mb 1\{\eta^N(t) \in \mc E^x_N\}\;.
\end{equation*}
In this formula, $\kappa$ represents the number of metastable sets and
$\mc E^x_N$, $1\le x\le \kappa$, the metastable sets. The
non-Markovian dynamics $X^N_t$ is called the order process or the
the order in short.

The main result of \cite{bl2,bl7} states that under certain
conditions, which can be expressed only in terms of the stationary
measure and of the capacities between the metastable sets, the order
converges in some time scale and in some topology to a Markov process
on $S=\{1, \dots, \kappa\}$.

The main drawback of the method \cite{bl2,bl7} is that it requires the
process to visit points. More precisely, we needed to assume that each
metastable set $\ms E^x_N$ contains a configuration $\xi^x_N$ which,
once the process enters $\ms E^x_N$, is visited before the process
reaches another metastable set:
\begin{equation}
\label{00}
\lim_{N\to\infty} \sup_{\eta\in \ms E^x_N} \bb P_{\eta} \big[ \, H_{
\breve{\ms E}^{x}_N}<H_{\xi^x_N} \,\big]\;=\; 0
\end{equation}
for all $x\in S$.  Here, $H_A$, $A\subset E_N$, stands for the hitting
time of $A$, $\breve{\ms E}^{x}_N = \cup_{y\not =x} \ms E^{y}_N$, and
$\bb P_\eta$ represents the distribution of the process $\eta^N(t)$
starting from the configuration $\eta$. The configuration $\xi^x_N$ is
by no means special. It is shown in \cite{bl2} that if this property
holds for one configuration $\xi$ in $\ms E^x_N$, it holds for any
configuration in $\ms E^x_N$.

Property \eqref{00} is fulfilled by some dynamics, as sticky
zero-range processes \cite{bl3, l1}, trap models \cite{jlt1, jlt2} or
Markov processes on finite sets \cite{bl4, bl5}, but it is clearly not
fulfilled in general.
 
The purpose of this paper is to replace condition \eqref{00} by
assumptions on the relaxation time of the process reflected at the
boundary of a metastable set. We propose two different set of
hypotheses. The first set essentially requires only the spectral gap
of the process to be much smaller than the spectral gaps of the
reflected processes on each metastable set, and the average jump rates
among the metastable sets to converge when properly
renormalized. Under these conditions, Theorem \ref{s02} states that
the finite-dimensional distributions of the order process converge to
the finite-dimensional distributions of a finite state Markov chain,
provided the initial distribution is not too far from the equilibrium
measure.

On the other hand, if one is able to show that the mixing times of the
reflected processes on each metastable set are much smaller than the
relaxation time of the process, Theorem \ref{s00} and Lemma \ref{s15}
affirm that the order process converges to a finite state Markov
chain. Hence, the condition that the process visits points is replaced
in this article by estimates on the mixing times of the reflected
processes.

In Section \ref{sec5}, we apply these results to two models. We show
that the polymer in the depinned phase considered by Caputo et al. in
\cite{cmt, clmst} satisfy the first set of conditions and that the dog graph
introduced by Diaconis and Saloff-Coste \cite{sc1} satisfy the second
set of assumptions. H. Lacoin and A. Teixeira \cite{lt1} are presently
working on another polymer model in which the second set of conditions
can be verified.

\section{Notation and results}
\label{sec0}

Fix a sequence $(E_N: N\ge 1)$ of countable state spaces. The elements
of $E_N$ are denoted by the Greek letters $\eta$, $\xi$. For each
$N\ge 1$ consider matrix $R_N : E_N\times E_N \to \bb R$ such that
$R_N(\eta, \xi)\ge 0$, $\eta\not = \xi$, $-\infty < R_N(\eta, \eta)
<0$, $\sum_{\xi} R_N(\eta,\xi)=0$, $\eta\in E_N$.  Denote by $\{\eta^N
(t) : t\ge 0\}$ the right-continuous, continuous-time strong Markov
process on $E_N$ whose generator $L_N$ is given by
\begin{equation}
\label{c01}
(L_Nf) (\eta) \,=\, \sum_{\xi\in E_N} R_N(\eta,\xi)
\, \big\{f(\xi)-f(\eta)\big\}\;,
\end{equation}
for bounded functions $f:E_N\to \bb R$. We assume that $\eta^N(t)$ is
positive-recurrent and reversible. Denote by $\pi=\pi_N$ the unique
invariant probability measure, by $\lambda_N(\eta)$, $\eta\in E_N$,
the holding rates, $\lambda_N(\eta) = \sum_{\xi\not = \eta}
R_N(\eta,\xi)$, and by $p_N(\eta,\xi)$, $\eta,\xi\in E_N$, the jump
probabilities: $p_N(\eta,\xi) = \lambda_N(\eta)^{-1} \, R_N(\eta,\xi)$
for $\eta\not = \xi$, and $p_N(\eta,\eta)=0$ for $\eta\in E_N$. We
assume that $ p_N(\eta,\xi)$ are the transition probabilities of a
positive-recurrent discrete-time Markov chain. In particular the
measure $M_N(\eta) := \pi_N(\eta) \lambda_N(\eta)$ is finite.

Throughout this article we omit the index $N$ as much as possible. We
write, for instance, $\eta(t)$, $\pi$ for $\eta^N(t)$, $\pi_N$,
respectively. Denote by $D(\bb R_+,E_N)$ the space of right-continuous
trajectories with left limits endowed with the Skorohod topology. Let
$\bb P_{\eta} = \bb P^N_{\eta}$, $\eta\in E_N$, be the probability
measure on $D(\bb R_+,E_N)$ induced by the Markov process $\{\eta (t)
: t\ge 0\}$ starting from $\eta$. Expectation with respect to $\bb
P_{\eta}$ is denoted by $\bb E_{\eta}$.

For a subset $\ms A$ of $E_N$, denote by $H_{\ms A}$ the hitting time
of $\ms A$ and by $H^+_{\ms A}$ the return time to $\ms A$:
\begin{equation}
\label{71}
\begin{split}
& H^+_{\ms A} \,=\, \inf\{ t>0 : \eta (t) \in \ms A \,,\, 
\eta (s) \not= \eta (0) \;\;\textrm{for some $0< s <
  t$}\}\,,  \\
&\quad H_{\ms A} \,:=\, \inf \big\{ t > 0 : 
\eta (t) \in \ms A \big\}\,,   
\end{split}
\end{equation}
with the convention that $H_{\ms A} = \infty$, $H^+_{\ms A} = \infty$
if $\eta (s)\not\in \ms A$ for all $s>0$. We sometimes write $H(\ms
A)$ for $H_{\ms A}$. Denote by $\Cap_N (\ms A, \ms B)$ the capacity
between two disjoint subsets $\ms A$, $\ms B$ of $E_N$:
\begin{equation*}
\Cap_N (\ms A, \ms B) \;=\;
\sum_{\eta\in \ms A} \pi(\eta)\, \lambda(\eta)\, \bb P_{\eta} 
\big[ H_{\ms B} < H_{\ms A}^+ \big]\;.
\end{equation*}

Denote by $L^2(\pi)$ the space of square summable functions $f:E_N\to
\bb R$ endowed with the scalar product $\<f,g\>_\pi = \sum_{\eta\in
  E_N} \pi(\eta) f(\eta) g(\eta)$. Let $\mf g=\mf g_N$ be the spectral
gap of the generator $L_N$:
\begin{equation*}
\mf g \;=\; \inf_f \frac{\< (-L_N) f, f\>_{\pi}} {\<f,f\>_{\pi}} \;,
\end{equation*}
where the infimum is carried over all functions $f$ in $L^2(\pi)$
which are orthogonal to the constants: $\<f,1\>_\pi=0$.

Fix a finite number of disjoint subsets $\ms E^1_N, \dots, \ms
E^\kappa_N$, $\kappa\ge 2$, of $E_N$: $\ms E^x_N\cap \ms
E^y_N=\varnothing$, $x\neq y$. The sets $\ms E^x_N$ have to be
interpreted as wells for the Markov dynamics $\eta(t)$.  Let $\ms
E_N=\cup_{x\in S}\ms E^x_N$ and let $\Delta_N=E_N \setminus \ms E_N$
so that
\begin{equation}
\label{nv1}
E_N \,=\, \ms E^1_N\cup\dots \cup \ms E^{\kappa}_N
\cup\, \Delta_N\;. 
\end{equation}
In contrast with the wells $\ms E^x_N$, $\Delta_N$ is a set of small
measure which separates the wells.

\smallskip\noindent{\bf A. Trace process}. Denote by $\{\eta^{\ms E}
(t): t\ge 0\}$ the $\ms E_N$-valued Markov process obtained as the
trace of $\{\eta^N (t): t\ge 0\}$ on $\ms E_N$. We refer to
\cite[Section 6.1]{bl2} for a precise definition.  The rate at which
the trace process jumps from $\eta$ to $\xi\in \ms E_N$ is denoted by
$R^{\ms E}(\eta, \xi)$ and its generator by $L_{\ms E}$:
\begin{equation*}
(L_{\ms E} f)(\eta) \,=\, \sum_{\xi\in \ms E_N} R^{\ms  E}(\eta,\xi)
\, \big\{f(\xi)-f(\eta)\big\}\;, \quad \eta\in \ms  E_N\;.
\end{equation*}
By \cite[Proposition 6.3]{bl2}, the probability measure $\pi$
conditioned to $\ms E_N$, $\pi_{\ms E} (\eta) = \pi(\eta)/\pi(\ms
E_N)$, is reversible for the trace process.

Let $\bb P^{\ms E}_{\eta}$, $\eta\in \ms E_N$, be the probability
measure on $D(\bb R_+, \ms E_N)$ induced by the trace process
$\{\eta^{\ms E} (t) : t\ge 0\}$ starting from $\eta$. Expectation with
respect to $\bb P^{\ms E}_{\eta}$ is denoted by $\bb E^{\ms
  E}_{\eta}$.  Denote by $\mf g_{\ms E}$ the spectral gap of the
trace process:
\begin{equation*}
\mf g_{\ms E} \;=\; \inf_f \frac{\< (-L_{\ms E}) f, f\>_{\pi_{\ms E}}} 
{\<f,f\>_{\pi_{\ms E}}} \;,
\end{equation*}
where the infimum is carried over all functions $f$ in $L^2(\pi_{\ms
  E})$ which are orthogonal to the constants: $\<f,1\>_{\pi_{\ms
    E}}=0$.

Proposition \ref{s05} presents an estimate of the spectral gap of the
trace process in terms of the spectral gap of the original process.

\begin{proposition}
\label{s05}
Let $f$ be an eigenfunction associated to $\mf g$ such that
$E_{\pi}[f^2] = 1$, $E_{\pi}[f] = 0$. Then,
\begin{equation*}
\mf g_{\ms E} \, \Big\{ 1 - \frac 1{\pi(\ms E_N)}\,
E_{\pi} \big[f^2 \mb 1\{\ms E^c_N\}\big] \Big\} 
\;\le\; \mf g \;\le\; \mf g_{\ms E} \;. 
\end{equation*}
\end{proposition}

In the examples we have in mind $\pi(\ms E_N)$ converges to $1$. In
particular, if we show that an eigenfunction associated to $\mf g$ is
bounded, $\mf g_{\ms E}/\mf g$ converges to $1$.  We provide in
Lemma \ref{s22} an upper bound for $\mf g_{\ms E}$ in terms of
capacities.

Denote by $\Psi_N:\ms E_N\mapsto S = \{1, \dots, \kappa\}$, the
projection given by
$$
\Psi_N(\eta) \;=\; \sum_{x=1}^\kappa  x \, 
\mathbf 1\{\eta \in \ms E^x_N\}\;.
$$
and by $\{X^N_t: t\ge 0\}$ the stochastic process on $S$ defined by
$X^N_t=\Psi_N(\eta^{\ms E}(t))$. Clearly, besides trivial cases,
$\{X^N_t: t\ge 0\}$ is not Markovian. We refer to $X^N_t$ as the
\emph{order} process or order for short.

\smallskip\noindent{\bf B. Reflected process.} Denote by $\{\eta^{\mb r,
  x} (t) : t\ge 0\}$, $1\le x \le\kappa$, the Markov process $\eta
(t)$ reflected at $\ms E^x_N$. This is the process obtained from the
Markov process $\eta(t)$ by forbiding all jumps from $\eta$ to $\xi$
if $\eta$ or $\xi$ do not belong to $\ms E^x_N$. The generator $L_{\mb
  r, x}$ of this Markov process is given by
\begin{equation*}
(L_{\mb r, x} f)(\eta) \,=\, \sum_{\xi\in \ms E^x_N} R_N(\eta,\xi)
\, \big\{f(\xi)-f(\eta)\big\}\;, \quad \eta\in \ms E^x_N\;.
\end{equation*}
Assume that the reflected process $\eta^{\mb r, x} (t)$ is irreducible
for each $1\le x\le \kappa$. It is easy to show that the conditioned
probability measure $\pi_{x}$ defined by 
\begin{equation}
\label{jo1}
\pi_{x} (\eta) \;=\; \frac{\pi(\eta)}{\pi(\ms E^x_N)}\;,
\quad \eta\in \ms E^x_N \;, 
\end{equation}
is reversible for the reflected process.  Let $\mf g_{\mb r, x}$ be
the spectral gap of the reflected process:
\begin{equation*}
\mf g_{\mb r, x} \;=\; \inf_f \frac{\< (-L_{\mb r, x}) f, f\>_{\pi_x}} 
{\<f,f\>_{\pi_x}} \;,
\end{equation*}
where the infimum is carried over all functions $f$ in $L^2(\pi_x)$
which are orthogonal to the constants: $\<f,1\>_{\pi_x}=0$. 

\smallskip\noindent{\bf C. Enlarged process}. Consider a irreducible,
positive recurrent Markov process $\xi(t)$ on a countable set $E$
which jumps from a state $\eta$ to a state $\xi$ at rate $R(\eta,
\xi)$. Denote by $\pi$ the unique stationary state of the process. Let
$E^\star$ be a copy of $E$ and denote by $\eta^\star\in E^\star$ the
copy of $\eta\in E$. Following \cite{bg11}, for $\gamma >0$ denote by
$\xi^\gamma (t)$ the Markov process on $E \cup E^\star$ whose jump
rates $R^\gamma (\eta,\xi)$ are given by
\begin{equation*}
R^\gamma (\eta,\xi) \;=\; 
\begin{cases}
R(\eta,\xi) & \text{if $\eta$ and $\xi\in E$,} \\
\gamma & \text{if $\xi = \eta^\star$ or if $\eta = \xi^\star$,} \\
0 & \text{otherwise.} 
\end{cases}
\end{equation*}
Therefore, being at some state $\xi^\star$ in $E^\star$, the process
may only jump to $\xi$ and this happens at rate $\gamma$. In contrast,
being at some state $\xi$ in $E$, the process $\xi^\gamma (t)$ jumps
with rate $R (\xi, \xi')$ to some state $\xi'\in E$, and jumps with
rate $\gamma$ to $\xi^\star$.  We call the process $\xi^\gamma (t)$
the \emph{$\gamma$-enlargement} of the process $\xi(t)$.

Let $\pi_\star$ be the probability measure on $E \cup E^\star$ defined
by
\begin{equation*}
\pi_\star (\eta) = (1/2)\, \pi (\eta)\; ,\; \quad
\pi_\star (\eta^\star) = \pi_\star (\eta) \;,\;\; 
\eta\in E\;.
\end{equation*}
The probability measure $\pi_\star$ is invariant for the enlarged
process $\xi^\gamma (t)$ and is reversible whenever $\pi$ is
reversible.  \smallskip

Let $\ms E^{\star, x}_N$, $1\le x\le \kappa$, be a copy of the set
$\ms E^x_N$ and let $\ms E^\star_N = \cup_{1\le x\le \kappa} \ms
E^{\star, x}_N$, $\breve{\ms E}^{\star, x}_N = \cup_{y\not = x} \ms
E^{\star, y}_N$.  Fix a sequence $\gamma = \gamma_N$ and denote by
$\eta^\star (t) = \eta^{\ms E, \gamma}$ the $\gamma$-enlargement of
the trace process $\eta^{\ms E}(t)$.  Denote the generator of this
Markov chain by $L_\star$, by $R_\star (\eta,\xi)$ the rate at which
it jumps from $\eta$ to $\xi$, and by $\lambda_\star(\eta)$ the
holding rates, $\lambda_\star(\eta) = \sum_{\xi\in \ms E_N \cup \ms
  E^\star_N} R_\star (\eta,\xi)$.

Denote by $\bb P^{\star, \gamma}_\eta$, $\eta\in \ms E_N \cup \ms
E^\star_N$, the probability measure on the path space $D(\bb R_+, \ms
E_N \cup \ms E^\star_N)$ induced by the Markov process $\eta^\star(t)$
starting from $\eta$ and recall the definition of the hitting time and
the return time introduced in \eqref{71}.  For $x\not = y\in S$, let
$r_N(x,y)$ be the average rate at which the enlarged process
$\eta^\star (t)$ jumps from $\ms E^{\star, x}_N$ to $\ms E^{\star,
  y}_N$:
\begin{equation}
\label{31}
\begin{split}
r_N(x,y) \; &=\; \frac{1}{\pi_{\star} (\ms E^{\star, x}_N)}
\sum_{\eta\in \ms E^{\star, x}_N} \pi_{\star} (\eta)\, \lambda_\star(\eta)
\, \bb P^{\star, \gamma}_\eta
\big[ H_{\ms E^{\star, y}_N} < H^+_{\breve{\ms E}^{\star, y}_N} \big] \\
&=\; \frac{\gamma}{\pi_{\ms E} (\ms E^{x}_N)}
\sum_{\eta\in \ms E^{x}_N} \pi_{\ms E} (\eta)\, \bb P^{\star, \gamma}_\eta
\big[ H_{\ms E^{\star, y}_N} < H_{\breve{\ms E}^{\star, y}_N} \big] \;.  
\end{split}
\end{equation}
By \cite[Proposition 6.2]{bl2}, $r_N(x,y)$ corresponds to the average
rate at which the trace of the process $\eta^\star(t)$ on $\ms
E^\star_N$ jumps from $\ms E^{\star, x}_N$ to $\ms E^{\star,
  y}_N$. This explains the terminology.

For two disjoint subsets $\ms A$, $\ms B$ of $\ms E_N \cup \ms
E^\star_N$, denote by $\Cap_\star(\ms A, \ms B)$ the capacity between
$\ms A$ and $\ms B$:
\begin{equation*}
\Cap_\star(\ms A, \ms B) \;=\; \sum_{\eta\in \ms A} \pi_\star(\eta) \,
\lambda_\star(\eta)\, \bb P^{\star,\gamma}_\eta \big[ H_{\ms B} <
H^+_{\ms A}\big]\;. 
\end{equation*}
Let $A$, $B$ be two disjoint subsets of $S$. Taking $\ms A= \cup_{x\in
  A} \ms E^{\star, x}_N$, $\ms B= \cup_{y\in B} \ms E^{\star, y}_N$ in
the previous formula, since the enlarged process may only jump from
$\eta^\star$ to $\eta$ and since $\pi_\star(\eta^\star) =
\pi_\star(\eta) = (1/2) \pi_{\ms E}(\eta)$,
\begin{equation}
\label{20}
\Cap_\star \Big( \bigcup_{x\in A} \ms E^{\star, x}_N, \bigcup_{y\in B} \ms
E^{\star, y}_N \Big) \;=\; \frac{\gamma} 2 
\sum_{x\in A} \sum_{\eta\in \ms E^{x}_N} \pi_{\ms E} (\eta) \,
\bb P^{\star,\gamma}_\eta \big[ H (\bigcup_{y\in B}
\ms E^{\star, y}_N) < H (\bigcup_{x\in A} \ms E^{\star, x}_N) \big]\;. 
\end{equation}
It follows from this identity and some simple algebra that
\begin{equation}
\label{28}
\pi_\star (\ms E^{x}_N) \, \sum_{y\not = x} r_N(x,y) \;=\;
\frac{\gamma} 2 \sum_{\eta\in \ms E^{x}_N} \pi_{\ms E} (\eta) \,
\bb P^{\star,\gamma}_\eta \big[ H_{\breve{\ms E}^{\star, x}_N}
< H_{\ms E^{\star, x}_N} \big] \;=\; 
\Cap_\star \big( \ms E^{\star, x}_N, \breve{\ms E}^{\star, x}_N \big) \;.
\end{equation}

\smallskip\noindent{\bf D. $L^2$ theory.}  We show in this subsection
that with very few assumptions one can prove the convergence of the
finite-dimensional distributions of the order $X^N_t$.  Let
\begin{equation}
\label{38}
\ms M_x \;=\; \min\big\{ \pi_{\ms E} (\ms E^{x}_N) \,,\,
1- \pi_{\ms E} (\ms E^{x}_N) \big\}\;, \quad x\in S\;.
\end{equation}

\begin{theorem}
\label{s02}
Suppose that there exist a non-negative sequence $\{\theta_N : N\ge
1\}$ and non-negative numbers $r(x,y)$, $x\not = y\in S$, such that
\begin{equation}
\tag*{\bf (L1)}
\begin{split}
\theta^{-1}_N \; & \ll\; \min_{x\in S} \mf g_{\mb r, x} \;, \\
\lim_{N\to \infty} \theta_N \, r_N (x,y) \; & =\; r(x,y)\;, \quad
x\not = y\in S\;.
\end{split}
\end{equation}
Fix $x_0\in S$. Let $\{\nu_N : N\ge 1\}$ be a sequence of probability
measures concentrated on $\ms E^{x_0}_N$, $\nu_N(\ms E^{x_0}_N)=1$,
and such that
\begin{equation}
\label{32}
\tag{\bf L2G}
E_{\pi_{\ms E}} \Big[ \Big( \frac {d\nu_N}{d\pi_{\ms E}} 
\Big)^2 \Big] \;\le\; \frac{C_0}{ \max_{x\in S} \, \ms M_x}
\end{equation}
for some finite constant $C_0$. Then, under $\bb P^{\ms E}_{\nu_N}$
the finite-dimensional distributions of the time-rescaled order ${\mb
  X}^N_t = X^N_{t \theta_N}$ converge to the finite-dimensional
distributions of the Markov process on $S$ which starts from $x_0$ and
jumps from $x$ to $y$ at rate $r(x,y)$.
\end{theorem}

Let $\nu_N$ be the measure $\pi_{x_0}$ defined in \eqref{jo1}. In this
case condition \eqref{32} becomes
\begin{equation}
\label{69}
\tag{\bf L2}
\min_{x\in S}\, \max_{z\not = x} \, \pi (\ms E^{z}_N)  
\;\le\; C_0 \, \min_{y\in S} \,  \pi (\ms E^{y}_N)
\end{equation}
for some finite constant $C_0$.  Condition \eqref{32} is satisfied in
two cases. Either if all wells $\ms E^y_N$ are stable sets (there
exists a positive constant $c_0$ such that $\pi_{\ms E}(\ms E^y_N) \ge
c_0$ for all $y\in S$, $N\ge 1$), or if there is only one stable set
and all the other ones have comparable measures (there exists $x\in S$
and $C_0$ such that $\lim_N \pi_{\ms E}(\ms E^x_N) =1$ and $\pi_{\ms
  E}(\ms E^y_N) \le C_0 \pi_{\ms E}(\ms E^z_N)$ for all $y$, $z\not =
x$). In particular, when there are only two wells, $|S|=2$, assumption
\eqref{32} is satisfied by the measures $\nu_N = \pi_{x}$,
$x=1,2$. \smallskip

Theorem \ref{s02} describes the asymptotic evolution of the trace of
the Markov $\eta(t)$ on $\ms E_N$. The next lemma shows that in the
time scale $\theta_N$ the time spent on the complement of $\ms E_N$ is
negligible.

\begin{lemma}
\label{s09}
Assume that
\begin{equation}
\tag*{\bf (L3)}
\lim_{N\to\infty} \frac{\pi(\Delta_N)}{\pi(\ms E^x_N)} \;=\; 0
\end{equation}
for all $x\in S$. Let $\{\nu_N : N\ge 1\}$ be a sequence of
probability measures concentrated on some well $\ms E^{x_0}_N$,
$x_0\in S$, and satisfying \eqref{32}. Then, for every $t>0$,
\begin{equation}
\label{58}
\lim_{N\to\infty} \bb E_{\nu_N}
\Big[\int_0^t {\bs 1}\{\eta (s\theta_N) \in \Delta_N\}
\, ds \,\Big]\, =\,0\,.
\end{equation}
\end{lemma}

\smallskip\noindent{\bf E. Mixing theory.}  If one is able to show
that the process mixes inside each well before leaving the well, the
assumptions on the initial state can be relaxed and the convergence of
the order can be derived.  Let $T^{\rm mix}_{\mb r, x}$, $x\in
S$, be the mixing time of the reflected process $\eta^{\mb r, x} (t)$.

\begin{theorem} 
\label{s00}
Fix $x_0\in S$.  Suppose that there exist a non-negative sequence
$\{\theta_N : N\ge 1\}$ and non-negative numbers $r(x,y)$, $x\not =
y\in S$, satisfying conditions {\rm {\bf (L1)}} and such that
\begin{equation}
\label{43}
\limsup_{N\to\infty} \theta_N \, E_{\pi_{x_0}}[R^{\ms E}(\eta,
\breve{\ms E}^{x_0})] \;<\; \infty\;.
\end{equation}
Let $\{\nu_N : N\ge 1\}$ be a sequence of probability measures
concentrated on $\ms E^{x_0}_N$, $\nu_N(\ms E^{x_0}_N)=1$. Assume that
condition \eqref{69} is fulfilled and that there exists a sequence
$T_N$, $T^{\rm mix}_{\mb r, x_0} \ll T_N \ll\theta_N$, such that
\begin{equation}
\tag*{\bf (L4)}
\lim_{N\to\infty} 
\bb P^{\ms E}_{\nu_N} \big[ H_{\breve{\ms E}^{x_0}_N} \le T_N
\;\big] \;=\; 0\;.
\end{equation}
Then, the finite-dimensional distributions of the time-rescaled order
${\mb X}^N_t = X^N_{t \theta_N}$ under $\bb P^{\ms E}_{\nu_N}$
converges to the finite-dimensional distributions of the Markov
process on $S$ which starts from $x_0$ and jumps from $x$ to $y$ at
rate $r(x,y)$.
\end{theorem}

Assumption \eqref{43} is not difficult to be verified. By \cite[Lemma
6.7]{bl2},
\begin{equation}
\label{61}
E_{\pi_x} \big[ R^{\ms E}(\eta, \breve{\ms E}^{x}_N) \big] \;=\;
\frac 1{\pi(\ms E^{x}_N)} \; \Cap_N (\ms E^{x}_N , \breve{\ms
  E}^{x}_N)\; ,
\end{equation} 
The Dirichlet principle \cite{g1, gl2} provides a variational formula
for the capacity and a bound for the expression in \eqref{43}. We show
in \eqref{50} below that $\sum_{y\not = x} r_N(x,y) \le E_{\pi_x}
[ R^{\ms E}(\eta, \breve{\ms E}^{x}_N)]$.

An assumption slightly stronger than ({\bf L4}) gives tightness of the
speeded-up order.  For a probability measure $\nu_N$ on $\ms
E_N$, denote by $\bb Q_{\nu_N}$ the probability measure on the path
space $D(\bb R_+, S)$ induced by the time-rescaled order ${\mb
  X}^N_{t} = \Psi_N(\eta^{\ms E}(t\theta_N))$ starting from $\nu_N$.

\begin{lemma}
\label{s15}
Let $\{\theta_N : N\ge 1\}$ be a sequence such that $\theta^{-1}_N \ll
\min_{x\in S} \mf g_{\mb r, x}$ and such that for all $x\in S$,
\begin{equation}
\label{51}
\limsup_{N\to\infty} \theta_N \, E_{\pi_x}[R^{\ms E}(\eta,
\breve{\ms E}^{x}_N)] \;<\; \infty\;.
\end{equation}
Assume that there exists a sequence $T_N$ such that
$\max_{x\in S} T^{\rm mix}_{\mb r, x} \ll T_N$ and such that
for all $x\in S$,
\begin{equation}
\tag*{\bf (L4U)}
\lim_{N\to\infty} \sup_{\eta\in \ms E^{x}_N}
\bb P^{\ms E}_{\eta} \big[ H_{\breve{\ms E}^{x}_N} \le T_N
\;\big] \;=\; 0\;.
\end{equation}
Let $\nu_N$ be a sequence of probability measures on $\ms E_N$.
Then, the sequence $(\bb Q_{\nu_N} : N\ge 1)$ is tight.
\end{lemma}

In Section \ref{sec6} we present a bound for the probability appearing
in condition ({\bf L4U}).  Let $\ms F^x_N$, $x\in S$, be subsets of
$E_N$ containing $\ms E^x_N$, $\ms E^x_N \subset \ms F^x_N$. Denote by
$T^{\rm mix}_{\mb r, \ms F^x_N}$ the mixing time of the process
$\eta(t)$ reflected at $\ms F^x_N$.

\begin{lemma}
\label{s25}
Fix $x\in S$ and suppose that there exist a set $\ms D^x_N\subset \ms
E^x_N$ and a sequence $T_N$, $T^{\rm mix}_{\mb r, \ms F^x_N} \ll T_N
\ll\theta_N$, such that
\begin{equation}
\tag*{\bf (L4E)}
\lim_{N\to\infty}  \max_{\eta\in \ms D^x_N} 
\bb P_{\eta} \big[ H_{(\ms F^{x}_N)^c} \le T_N \;\big] \;=\; 0 \;.
\end{equation}
Then, \eqref{58} holds for any $t>0$ and any sequence of probability
measures $\nu_N$ concentrated on $\ms D^x_N$ provided condition {\rm
  ({\bf L3})} is in force.
\end{lemma}

Even if we are not able to prove the pointwise versions {\bf (L4U)} or
{\bf (L4E)} of the mixing condition, we can still show that the
measures of the wells converge in the Cesaro sense.

\begin{proposition}
\label{s16} 
Fix $x_0\in S$.  Assume that conditions {\rm {\bf (L1)}}, \eqref{69}
and \eqref{43} are fulfilled.  Let $\{\nu_N : N\ge 1\}$ be a sequence
of probability measures concentrated on $\ms E^{x_0}_N$ and satisfying
conditions \eqref{58} and {\rm {\bf (L4)}}. Denote by $\{S_N(r) \,|\,
r\ge 0\}$ the semigroup of the process $\eta(r)$.  Then, for every
$t>0$ and $x\in S$,
\begin{equation}
\label{56}
\lim_{N\to\infty} \int_0^t  \big[ \nu_N S_N(\theta_N r) \big] \, 
( \ms E^{x}_N ) \, dr \;=\; \int_0^t [\delta_{x_0} S (r)](x) \, dr \;, 
\end{equation}
where $S (r)$ stands for the semigroup of the continuous-time Markov
chain on $S$ which jumps from $y$ to $z$ at rate $r(y,z)$, and where
$\delta_{x_0}$ stands for the probability measure on $S$ concentrated
at $x_0$.
\end{proposition}

\smallskip\noindent{\bf F. Two valleys.}  We suppose from now on that
there are only two valleys, $\ms E^1_N = \ms A_N$ and $\ms E^2_N = \ms
B_N$. It is possible in this case to establish a relation between the
spectral gap of the trace process and the capacities of the enlarged
process, and to re-state Theorems \ref{s02} and \ref{s00} in a simpler
form. Assume that the sets $\ms E^x_N$, $x=1,2$, have an asymptotic
measure:
\begin{equation*}
\lim_{N\to \infty} \pi_{\ms E}(\ms E^x_N) \;=\; m(x) \;, 
\end{equation*}
and suppose, to fix ideas, that $m(1) \le m(2)$.

\begin{theorem}
\label{s07}
Assume that $\mf g_{\ms E} \ll \min_{x=1,2} \mf g_{\mb r, x}$ and
consider a sequence $\gamma_N$ such that $\mf g_{\ms E} \ll
\gamma_N \ll \min_{x=1,2} \mf g_{\mb r, x}$. Then,
\begin{equation*}
\lim_{N\to\infty} \frac{\Cap_\star (\ms A^\star_N, \ms B^\star_N)}
{\mf g_{\ms E}\, \pi_{\ms E}(\ms A_N) \, \pi_{\ms E}(\ms B_N)} 
\;=\; \frac 12 \;\cdot
\end{equation*}
\end{theorem}

This result follows from \cite[Theorem 2.12]{bg11}.  Under the
assumptions of Proposition \ref{s05} we may replace in this statement the
spectral gap of the trace process by the spectral gap of the original
process. Moreover, in view of \eqref{28}, 
\begin{equation}
\label{72}
\lim_{N\to\infty} \mf g_{\ms E}^{-1} \, r_N(x,y) \;=\; m(y)\;.  
\end{equation}

When there are only two valleys, the right hand side of equation
\eqref{32} is equal to $C_0 \min \{\pi_{\ms E} (\ms E^1_N), \pi_{\ms
  E} (\ms E^2_N) \}^{-1}$.  Condition \eqref{32} then becomes
\begin{equation}
\label{39}
E_{\pi_{\ms E}} \Big[ \Big( \frac {d\nu_N}{d\pi_{\ms E}} 
\Big)^2 \Big] \;\le\; \max_{x=1,2} \, \frac{C_0}{\pi_{\ms E} (\ms E^x_N)} 
\end{equation}
for some finite constant $C_0$.  The measures $\nu_N = \pi_1$, $\pi_2$
clearly fulfill this condition.  We summarize in the next lemma the
observations just made.

\begin{lemma}
\label{s26}
Suppose that there are only two wells, $S=\{1,2\}$, and set $\theta_N =
\mf g^{-1}_{\ms E}$. Then, condition {\rm ({\bf L1})} is reduced to
the condition that
\begin{equation}
\tag*{\bf (L1B)}
\mf g_{\ms E} \;\ll\; \min \{ \mf g_{\mb r, 1} \,,\, 
\mf g_{\mb r, 2}\}  \;,
\end{equation}
the asymptotic rates $r(x,y)$ are given by $r(x,y)=m(y)$, and
condition \eqref{69} is always in force.
\end{lemma}

In the case of two wells, there are two different asymptotic
behaviors. Assume first that $m(1) >0$. In this case $\ms E^1_N$ and
$\ms E^2_N$ are stable sets and ${\mb X}^N_t$ jumps asymptotically
from $x$ to $3-x$ at rate $m(3-x)$. If $m(1) =0$ and if $\nu_N$ is a
sequence of measures concentrated on $\ms E^1_N$, $\ms E^1_N$ is a
metastable set, $\ms E^2_N$ a stable set, and ${\mb X}^N_t$ jumps
asymptotically from $1$ to $2$ at rate $1$, remaining forever at $2$
after the jump.

\begin{remark}
\label{s06}
The average rates $r_N(x,y)$ introduced in \eqref{31} are different
from those which appeared in \cite{bl2}, but can still be expressed in
terms of the star-capacities:
\begin{equation*}
\begin{split}
& \pi_\star (\ms E^{x}_N) \, r_N(x,y) \;=\; \\
&\quad \frac 12 \, \Big\{\Cap_\star(\ms
E^{\star, x}_N, \breve{\ms E}^{\star, x}_N ) +
\Cap_\star(\ms E^{\star, y}_N, \breve{\ms E}^{\star, y}_N ) 
- \Cap_\star(\ms E^{\star, x}_N \cup \ms E^{\star, x}_N , 
\cup_{z\not = x,y} \ms E^{\star, z}_N) \Big\} \;. 
\end{split}
\end{equation*}
\end{remark}

To prove this identity observe that by \eqref{20} the right hand side
is equal to
\begin{equation*}
\frac{\gamma}{4}
\sum_{\eta\in \ms E^{x}_N} \pi_{\ms E} (\eta)\, \bb P^{\star, \gamma}_\eta
\big[ H_{\ms E^{\star, y}_N} < H_{\breve{\ms E}^{\star, y}_N} \big]
\;+\;
\frac{\gamma}{4}
\sum_{\eta\in \ms E^{y}_N} \pi_{\ms E} (\eta)\, \bb P^{\star, \gamma}_\eta
\big[ H_{\ms E^{\star, x}_N} < H_{\breve{\ms E}^{\star, x}_N} \big]\;.
\end{equation*}
By \eqref{31}, the first term is equal to $(1/2) \pi_\star (\ms
E^{x}_N) r_N(x,y)$. By definition of the enlarged process, the second
term can be written as
\begin{equation*}
\frac{\gamma}{2}
\sum_{\eta\in \ms E^{\star, y}_N} \pi_{\star} (\eta)\, 
\bb P^{\star, \gamma}_\eta
\big[ H_{\ms E^{\star, x}_N} = H^+_{\ms E^{\star}_N} \big]\;. 
\end{equation*}
By reversibility, $\pi_{\star} (\eta)\, \bb P^{\star, \gamma}_\eta [
H_\xi = H^+_{\ms E^{\star}_N} ] = \pi_{\star} (\xi)\, \bb P^{\star,
  \gamma}_\xi [ H_\eta = H^+_{\ms E^{\star}_N} ]$, $\eta\in \ms
E^{\star, y}_N$, $\xi\in \ms E^{\star, x}_N$. This concludes the proof
of the remark.

\smallskip

We conclude this section pointing out an interesting difference
between Markov processes exhibiting a metastable behavior and a Markov
processes exhibiting the cutoff phenomena \cite{lpw1}. On the level of
trajectories, after remaining a long time in a metastable set, the
first ones perform a sudden transition from one metastable set to
another, while on the level of distributions, as stated in Proposition
\ref{s16} below, in the relevant time scale these processes relax
smoothly to the equilibrium state. In contrast, processes exhibiting
the cutoff phenomena do not perform sudden transitions on the path
level, but do so on the distribution level, moving quickly in a
certain time scale from far to equilibium to close to equilibrium.

\section{Convergence of the finite-dimensional distributions}
\label{sec1}

We start this section with an important estimate which allows the
replacement of the time integral of a function $f: \ms E_N \to\bb R$
by the time integral of the conditional expectation of $f$ with
respect to the $\sigma$-algebra generated by the partition $\ms E^1_N,
\dots, \ms E^\kappa_N$.

\subsection{Replacement Lemma}

Denote by $\Vert f\Vert_{-1}$ the $\mc H_{-1}$ norm associated to the
generator $L_{\ms E}$ of a function $f: \ms E_N\to \bb R$ which has
mean zero with respect to $\pi_{\ms E}$:
\begin{equation*}
\Vert f \Vert^2_{-1}
\;=\; \sup_{h} \Big\{ 2 \< f , h \>_{\pi_{\ms E}} 
- \< h, (- L_{\ms E}) h \>_{\pi_{\ms E}} \Big\}\; ,
\end{equation*}
where the supremum is carried over all functions $h:\ms E_N\to \bb R$
with finite support. By \cite[Lemma 2.4]{klo1}, for every function $f:
\ms E_N \to \bb R$ which has mean zero with respect to $\pi_{\ms E}$,
and every $T>0$,
\begin{equation}
\label{10}
\bb E^{\ms E}_{\pi_{\ms E}} \Big[ \sup_{0\le t\le T}
\Big( \int_0^{t} f(\eta^{\ms E} (s)) 
\, ds  \, \Big)^2 \, \Big] \;\le\; 24\, T\, \Vert f \Vert^2_{-1}\;.
\end{equation}

Similarly, for a function $f: \ms E^N_x\to\bb R$ which has mean zero
with respect to $\pi_x$, denote by $\Vert f\Vert_{x,-1}$ the $\mc
H_{-1}$ norm of $f$ with respect to the generator $L_{\mb r, x}$ of
the reflected process at $\ms E^x_N$:
\begin{equation}
\label{02}
\Vert f \Vert^2_{x,-1}
\;=\; \sup_{h} \Big\{ 2 \< f , h \>_{\pi_x} 
- \< h, (-L_{\mb r, x}) h \>_{\pi_x} \Big\}\; ,
\end{equation}
where the supremum is carried over all functions $h:\ms E^x_N\to \bb
R$ with finite support. It is clear that
\begin{equation*}
\sum_{x\in S} \pi_{\ms E}(\ms E^x_N) \, \< h, (-L_{\mb r, x}) h \>_{\pi_x} 
\;\le\; \< h, (- L_{\ms E}) h \>_{\pi_{\ms E}}
\end{equation*}
for any function $h: \ms E_N \to\bb R$ with finite support. Note that
the generator of the trace process $L_{\ms E}$ may have jumps from the
boundary of a set $\ms E^N_x$ to its boundary which do not exist in
the original process. There are therefore two types of contributions
which appear on the right hand side but do not on the left hand
side. These ones, and jumps from one set $\ms E^x_N$ to another.  It
follows from the previous inequality that for every function $f:\ms
E_N\to \bb R$ which has mean zero with respect to each measure
$\pi_x$,
\begin{equation}
\label{03}
\Vert f \Vert^2_{-1} \;\le\; \sum_{x\in S} \pi_{\ms E}(\ms E^x_N)
\Vert f \Vert^2_{x,-1}\;.
\end{equation}

\begin{proposition}
\label{s01}
Let $\{\nu_N : N\ge 1\}$ be a sequence of probability measures on $\ms
E_N$. Then, for every function $f:\ms E_N\to \bb R$ which has mean
zero with respect to each measure $\pi_x$ and for every $T>0$,
\begin{equation*}
\Big( \bb E^{\ms E}_{\nu_N} \Big[ \sup_{t\le T} \, \Big|\int_0^{t} 
f(\eta^{\ms E}(s)) \, ds \Big| \,\Big] \Big)^2
\;\le\; 24 \, T\, 
E_{\pi_{\ms E}} \Big[ \Big( \frac {\nu_N}{\pi_{\ms E}} \Big)^2 \Big] 
\, \sum_{x\in S} \pi_{\ms E}(\ms E^x_N) \Vert f \Vert^2_{x,-1} \; .
\end{equation*}
\end{proposition}

\begin{proof}
By Schwarz inequality, the expression on the left hand side of the
previous displayed equation is bounded above by
\begin{equation*}
E_{\pi_{\ms E}} \Big[ \Big( \frac {\nu_N}{\pi_{\ms E}} \Big)^2 \Big] 
\,  \bb E^{\ms E}_{\pi_{\ms E}} \Big[ \sup_{t\le T} \Big( \int_0^{t} 
f (\eta^{\ms E}(s)) \,  ds \Big)^2 \,\Big] \;.
\end{equation*}
By \eqref{10} and by \eqref{03}, the second expectation is bounded by
\begin{equation*}
24 \, T\, \sum_{x\in S} \pi_{\ms E}(\ms E^x_N) 
\, \Vert f \Vert^2_{x,-1} \; ,
\end{equation*}
which concludes the proof of the proposition.
\end{proof}

By the spectral gap, for any function $f: \ms E^x_N \to \bb R$ which
has mean zero with respect to $\pi_x$, $\Vert f \Vert^2_{x,-1} \;\le\;
\mf g^{-1}_{\mb r,x} \, \<f,f\>_ {\pi_x}$. The next result follows
from this observation and the previous proposition.


\begin{corollary}
\label{s08}
Let $\{\nu_N : N\ge 1\}$ be a sequence of probability measures on $\ms
E_N$. Then, for every function $f:\ms E_N\to \bb R$ which has mean
zero with respect to each measure $\pi_x$ and for every $T>0$,
\begin{equation*}
\Big( \bb E^{\ms E}_{\nu_N} \Big[ \sup_{t\le T} \, \Big|\int_0^{t} 
f(\eta^{\ms E}(s)) \, ds \Big| \,\Big] \Big)^2
\;\le\; 24 \, T\, 
E_{\pi_{\ms E}} \Big[ \Big( \frac {\nu_N}{\pi_{\ms E}} \Big)^2 \Big] 
\, \sum_{x\in S} \pi_{\ms E}(\ms E^x_N) \, \mf g^{-1}_{\mb r,x}
\, \<f , f\>_ {\pi_x}  \; .
\end{equation*}
\end{corollary}

We have seen in \eqref{28} that $\Cap_\star ( \ms E^{\star, x}_N,
\breve{\ms E}^{\star, x}_N ) = \pi_\star (\ms E^{x}_N) \, \sum_{y\not
  = x} r_N(x,y)$. For similar reasons, $\Cap_\star \big( \ms E^{\star,
  x}_N, \breve{\ms E}^{\star, x}_N \big) = \sum_{y\not = x} \pi_\star
(\ms E^{y}_N) \, r_N(y,x)$. If $\theta_N r_N(x,y)$ converges, as
postulated in assumption ({\bf L1}), we obtain from these identities
that
\begin{equation}
\label{37}
\Cap_\star ( \ms E^{\star, x}_N, \breve{\ms E}^{\star, x}_N ) 
\;\le\; C_0 \,\theta^{-1}_N \, \ms M_x 
\end{equation}
for some finite constant $C_0$, where $\ms M_x$ has been introduced in
\eqref{38}.

\smallskip\noindent{\bf The equilibrium potentials.}  Fix a sequence
$\gamma=\gamma_N$ such that $\theta^{-1}_N \ll \gamma \ll \min_{x\in
  S} \mf g_{\mb r, x}$ and recall that we denote by $\eta^\star(t)$
the $\gamma$-enlargement of the trace process $\eta^{\ms E}(t)$.
Denote by $V_x$, $x\in S$, the equilibrium potential between the sets
$\ms E^{\star, x}_N$ and $\breve{\ms E}^{\star, x}_N$, $V_x(\eta) =
\bb P^{\star, \gamma}_\eta [ H_{\ms E^{\star, x}_N} < H_{\breve{\ms
    E}^{\star, x}_N} ]$. Since $L_\star V_x =0$ on $\ms E_N$, we
deduce that
\begin{equation}
\label{19}
\begin{split}
& (L_{\ms E} V_x)(\eta) \;=\; - \gamma\, [1-V_x(\eta)]\;,\;\; \eta\in
\ms E^{x}_N \;, \\
&\quad (L_{\ms E} V_x)(\eta) \;=\; \gamma \, V_x(\eta) \;,\;\;
\eta\in  \breve{\ms E}^{x}_N\;.
\end{split}
\end{equation}
Moreover, since $\pi_\star (\eta) = (1/2) \pi_{\ms E}(\eta)$, $\eta\in
\ms E_N$,
\begin{equation}
\label{24}
\begin{split}
& \Cap_\star(\ms E^{\star, x}_N, \breve{\ms E}^{\star, x}_N) \;=\; \\
& \quad \frac 12\Big\{ \gamma \sum_{\eta\in \ms E^{x}_N} \pi_{\ms
  E}(\eta) \, [1- V_x(\eta)]^2 
\;+\; \< (-L_{\ms E}) V_x, V_x\>_{\pi_{\ms E}} \;+\;
\gamma \sum_{\eta\in \breve{\ms E}^{x}_N} 
\pi_{\ms E}(\eta) \, V_x(\eta)^2 \Big\} \;.
\end{split}
\end{equation}
By assumption ({\bf L1}), for all $x\not = y\in S$, $r_N(x,y) \le C_0
\theta^{-1}_N$ for some finite constant $C_0$ and for all $N$ large
enough. Hence, by \eqref{37} and by \eqref{24}, for all $x\in S$
\begin{equation}
\label{25}
\gamma \sum_{\eta\in \ms E^{x}_N} \pi_{\ms
  E}(\eta) \, [1- V_x(\eta)]^2 
\;+\; \< (-L_{\ms E}) V_x, V_x\>_{\pi_{\ms E}} \;+\;
\gamma \sum_{\eta\in \breve{\ms E}^{x}_N} 
\pi_{\ms E}(\eta) \, V_x(\eta)^2 \;\le\; 
\frac{C_0\, \ms M_x}{\theta_N} \;\cdot
\end{equation}

\smallskip\noindent{\bf Uniqueness of limit points.}  Recall the
definition of the measure $\bb Q_{\nu_N}$ introduced just before Lemma
\ref{s15}, and let $\mf L$ be the generator of the $S$-valued Markov
process given by
\begin{equation*}
(\mf L F) (x) \;=\; \sum_{y\in S} r(x,y) [F(y) - F(x)]\;.
\end{equation*}

\begin{proposition}
\label{s10}
Assume that the hypotheses of Theorem \ref{s02} are in force. Then,
the sequence $\bb Q_{\nu_N}$ has at most one limit point, the
probability measure on $D(\bb R_+, S)$ induced by the Markov process
with generator $\mf L$ starting from $x_0$.
\end{proposition}

\begin{proof}
To prove the uniqueness of limit points, we use the martingale
characterization of Markov processes.  Fix a function $F: S\to\bb R$
and a limit point $\bb Q_*$ of the sequence $\bb Q_{\nu_N}$. We claim 
that
\begin{equation}
\label{30}
M^F_t \;:=\; F(X_t) \;-\; F(X_0) \;-\; \int_0^t (\mf L F)(X_s)\, ds
\end{equation}
is a martingale under $\bb Q_*$.

Fix $0\le s< t$ and a bounded function $U:D(\bb R_+, S)\mapsto \bb R$
depending only on $\{X_r : 0\le r\le s\}$ and continuous for the
Skorohod topology. We shall prove that
\begin{equation}
\label{21}
\bb E_{\bb Q_*}\,\big[M^F_t U\big] \;=\; 
\bb E_{\bb Q_*}\,\big[M^F_s U\big]\;.
\end{equation}

Let $G(\eta) = \sum_{x\in S} F(x) V_x(\eta)$, $\eta\in \ms E_N$.
By the Markov property of the trace process $\eta^{\ms E}(t)$,
\begin{equation*}
M^N_t \;=\; G(\eta^{\ms E}(t\theta_N)) \;-\; G(\eta^{\ms E}(0))
\;-\; \int_0^{t \theta_N} (L_{\ms E} G) (\eta^{\ms E}(s))\, ds
\end{equation*}
is a martingale. Let $U^N := U(X^N_{\bs \cdot})$. As $\{M^N_t :
t\ge 0 \}$ is a martingale,
$$
\bb E^{\ms E}_{\nu_N} \big[M^N_t U^N\big]
\, = \, \bb E^{\ms E}_{\nu_N} \big[M^N_s U^N\big]
$$
so that
\begin{equation}
\label{22}
\bb E^{\ms E}_{\nu_N} \Big[ U^N \Big\{ G(\eta^{\ms E}(t\theta_N))
\;-\; G(\eta^{\ms E}(s \theta_N))
\;-\; \int_{s \theta_N}^{t \theta_N} (L_{\ms E} G) (\eta^{\ms E}(r))\, dr
\Big\} \Big ] \,=\, 0\,.
\end{equation}

\noindent{\bf Claim A:} For all $x\in S$,
\begin{equation*}
\lim_{N\to \infty} \sup_{t\ge 0}\, 
\bb E^{\ms E}_{\nu_N} \big[ \,| \mb 1_{\ms E^x_N} (\eta^{\ms E}
(t\theta_N)) - V_x (\eta^{\ms E}(t\theta_N))|\, \big] \;=\;0\;.
\end{equation*}
Indeed, denote by $S_{\ms E}(t)$, $t\ge 0$, the semigroup associated
to the trace process $\eta^{\ms E}(t)$, and by $h_t$ the Radon-Nikodym
derivative $d\nu_N S_{\ms E}(t)/ d\pi_{\ms E}$. It is well known that
$E_{\pi_{\ms E}}[h^2_t] \le E_{\pi_{\ms E}}[h^2_0]$.  Hence, by
Schwarz inequality, the square of the expectation appearing in the
previous displayed formula is bounded above by
\begin{equation*}
E_{\pi_{\ms E}} \Big[ \Big( \frac{d\nu_N}{d\pi_{\ms E}} \Big)^2 \Big]
E_{\pi_{\ms E}} \big[\,| \mb 1_{\ms E^x_N}  - V_x |^2\, \big]\;.
\end{equation*}
To conclude the proof of the claim it remains to recall the definition
of the sequence $\gamma$, the estimate \eqref{25} and the assumption
on the sequence of probability measures $\nu_N$.

It follows from Claim A that 
\begin{equation*}
\lim_{N\to \infty} \sup_{t\ge 0}\,
\bb E^{\ms E}_{\nu_N} \big[ \,| (F\circ\Psi) (\eta^{\ms E}
(t\theta_N)) - G (\eta^{\ms E}(t\theta_N))|\, \big] \;=\;0\;.
\end{equation*}
Therefore, by \eqref{22},
\begin{equation}
\label{23}
\lim_{N\to\infty} \bb E^{\ms E}_{\nu_N} \Big[ U^N \Big\{ \Delta_{s,t}F
\;-\; \int_{s \theta_N}^{t \theta_N} (L_{\ms E} G) (\eta^{\ms E}(r))\, dr
\Big\} \Big ] \,=\, 0\,.
\end{equation}
where $\Delta_{s,t}F = (F\circ\Psi) (\eta^{\ms E}(t\theta_N)) \;-\;
(F\circ\Psi) (\eta^{\ms E}(s \theta_N)) = F(X^N_{t\theta_N}) -
F(X^N_{s\theta_N})$.

\smallskip\noindent{\bf Claim B:} Denote by $\mc P$ the
$\sigma$-algebra generated by the partition $\ms E^z_N$, $z\in S$. For
all $T>0$, $x\in S$,
\begin{equation*}
\lim_{N\to \infty} \bb E^{\ms E}_{\nu_N} \Big[ \sup_{t\le T \theta_N}
\Big| \int_0^t  \Big\{ (L_{\ms E} V_x) (\eta^{\ms E} (s)) - 
E \big[ L_{\ms E} V_x \,\big|\, \mc P \big]  
(\eta^{\ms E}(s)) \Big\} \, ds \Big|\, \Big] \;=\;0\;.
\end{equation*}

By the assumption on the sequence $\nu_N$ and by Proposition
\ref{s01}, the square of the expectation appearing in the previous
formula is bounded by
\begin{equation}
\label{26}
\frac{C_0 \, T\, \theta_N}{\max_{z\in S} \ms M_z} \, 
\sum_{y\in S}  \pi_{\ms E}(\ms E^y_N) 
\, \Vert \overline{L_{\ms E} V_x} \Vert^2_{y,-1} 
\end{equation}
for some finite constant $C_0$, where $\overline{G}$ stands for
$G-E_{\pi_{y}}[G]$.  By \eqref{19}, on the set $\ms E^x_N$, $L_{\ms E}
V_x = - \gamma [1-V_x(\eta)]$. Hence, by the spectral gap o the
reflected process and by \eqref{25},
\begin{equation*}
\Vert \overline{L_{\ms E} V_x} \Vert^2_{x,-1} \;=\; 
\gamma^2\, \Vert \overline{1-V_x} \Vert^2_{x,-1} \;\le \; 
\frac{\gamma^2}{\mf g_{\mb r, x}} \, \Vert 1-V_x \Vert^2_{\pi_{x}}
\;\le\; \frac{C_0 \, \gamma\, \ms M_x}
{\pi_{\ms E} (\ms E^x_N)\, \mf g_{\mb r, x} \, \theta_N}
\end{equation*}
for some finite constant $C_0$. Similarly, since $L_{\ms E} V_x =
\gamma V_x(\eta)$ on the set $\ms E^y_N$, $y\not = x$,
\begin{equation*}
\Vert \overline{L_{\ms E} V_x} \Vert^2_{y,-1} 
\;\le\; \frac{C_0 \, \gamma\, \ms M_x}
{\pi_{\ms E} (\ms E^y_N)\, \mf g_{\mb r, y} \, \theta_N}\;\cdot
\end{equation*}
Therefore, the sum appearing in \eqref{26} is bounded by $C_0 T \,
|S|\, \gamma \, \max_{z\in S} \mf g^{-1}_{\mb r, z} $ which vanishes as
$N\uparrow\infty$ by definition of $\gamma$, proving Claim B.

It follows from \eqref{23} and Claim B that 
\begin{equation}
\label{27}
\lim_{N\to\infty} \bb E^{\ms E}_{\nu_N} \Big[ U^N \Big\{ \Delta_{s,t}F
\;-\; \int_{s}^{t} 
\theta_N \, E \big[ L_{\ms E} G \,\big|\, \mc P \big] 
(\eta^{\ms E}(r\theta_N))\, dr \Big\} \Big ] \,=\, 0\,.
\end{equation}

We affirm that
\begin{equation}
\label{29}
E \big[ L_{\ms E} G \,\big|\, \mc P \big] (\eta) \;=\; 
\sum_{x\in S} \mb 1\{\eta\in \ms E^x_N\} \sum_{y\in S} r_N(x,y) 
[F(y)-F(x)]\;.
\end{equation}
Indeed, by \eqref{19},
\begin{equation*}
E \big[ L_{\ms E} V_x \,\big|\, \mc P \big] \;=\;
\begin{cases}
\displaystyle 
- \gamma \sum_{\eta\in \ms E^x_N} \frac{\pi_{\ms E} (\eta)}
{\pi_{\ms E} (\ms E^x_N)}
\, \bb P^{\star, \gamma}_{\eta}\big[ H_{\breve{\ms E}^{\star, x}_N}  <
H_{\ms E^{\star, x}_N} \big]\;,  & \eta\in \ms E^{x}_N\;,\\
\displaystyle
\gamma \sum_{\eta\in \ms E^{y}_N} \frac{\pi_{\ms E} (\eta)}
{\pi_{\ms E} (\ms E^y_N)}
\, \bb P^{\star, \gamma}_{\eta}\big[ H_{\ms E^{\star, x}_N} <
H_{\breve{\ms E}^{\star, x}_N} \big] \;, & \eta\in \ms E^{y}_N \;,\; y\not
= x\;.
\end{cases}
\end{equation*} 
By \eqref{28}, on the set $\ms E^{x}_N$, $E [ L_{\ms E} V_x \, |\, \mc
P ] = -\, \sum_{y\not = x} r_N(x,y)$, and by \eqref{31}, on the set $\ms
E^{y}_N$, $E [ L_{\ms E} V_x \, |\, \mc P ] = r_N(y,x)$. To conclude
the proof of \eqref{29} it remains to recall the definition of $G$.

By \eqref{27}, \eqref{29} and by definition of $X^N_t$,
\begin{equation*}
\lim_{N\to\infty} \bb E^{\ms E}_{\nu_N} \Big[ U^N \Big\{ 
\Delta_{s,t}F 
\;-\; \int_{s}^{t}  \sum_{y\in S} \theta_N \, r_N(X^N_{r\theta_N},y) 
\, [F(y)-F(X^N_{r\theta_N})] \,  dr \Big\} \Big ] \,=\, 0\,.
\end{equation*}
Since $\Delta_{s,t}F = F(X^N_{t\theta_N})- F(X^N_{s\theta_N})$, since
$U$ has been assumed to be continuous for the Skorohod topology and
since $\bb Q_*$ is a limit point of the sequence $\bb Q_{\nu_N}$, by
assumption ({\bf L1})
\begin{equation*}
E_{\bb Q_*} \Big[ U \Big\{ F(X_t) - F(X_s)
\;-\; \int_{s}^{t}  \sum_{y\in S} r (X_r,y) 
\, [F(y)-F(X_r)] \,  dr \Big\} \Big ] \,=\, 0\,,
\end{equation*}
proving \eqref{30} and the proposition. 
\end{proof}

\smallskip\noindent{\bf Proof of Theorem \ref{s02}.}  The proof is
similar to the one of Proposition \ref{s10}. We prove the convergence
of the one-dimensional distributions. The extension to higher
dimensional distributions is clear. Fix a function $F:S\to \bb R$. We
claim that for every $T\ge 0$,
\begin{equation}
\label{48}
\limsup_{N\to\infty} \sup_{0\le s<t\le T} \Big| \, 
\bb E_{\nu_N} \Big[ F(X^N_{t\theta_N})- F(X^N_{s\theta_N})
-\int_s^t (\mf L F)(X^N_{r\theta_N}) \, dr \Big] \, \Big| \;=\; 0\;.
\end{equation}
Recall the definition of the function $G:\ms E_N \to\bb R$ introduced
in the proof of the previous proposition. By Claim A and since
$G(\eta^{\ms E}(t\theta_N)) - \int_0^t \theta_N (L_{\ms E}G)
(\eta^{\ms E}(s\theta_N)) \, ds$ is a martingale, to prove \eqref{48},
it is enough to show that
\begin{equation*}
\limsup_{N\to\infty} \sup_{0\le t\le T} \Big| \, 
\bb E_{\nu_N} \Big[ \int_0^t \theta_N (L_{\ms E}G) (\eta^{\ms
  E}(r\theta_N)) \, dr  
-\int_0^t (\mf L F)(X^N_{r\theta_N}) \, dr \Big] \, \Big| \;=\; 0\;.
\end{equation*}
By Claim B, by the identity \eqref{29} and by the definition of
$X^N_t$, the proof of \eqref{48} is further reduced to the proof that
\begin{equation*}
\limsup_{N\to\infty} \sup_{0\le t\le T} \Big| \, 
\bb E_{\nu_N} \Big[ \int_0^t (\mf L_N F)(X^N_{r\theta_N})\, dr  
-\int_0^t (\mf L F)(X^N_{r\theta_N}) \, dr \Big] \, \Big| \;=\; 0\;,
\end{equation*} 
where $(\mf L_N F)(x) = \sum_{y\in S} \theta_N \, r_N(x,y) 
\, [F(y)-F(x)]$. To conclude the proof of \eqref{48}, it remains to
recall assumption ({\bf L1}).

It follows from \eqref{48} that the sequence $f_N(t) = \bb E_{\nu_N} [
F(X^N_{t\theta_N})]$ is equicontinuous in any compact interval
$[0,T]$. Moreover, if $F$ is an eigenfunction of the operator $\mf L$
associated to an eigenvalue $\lambda$, all limit points $f(t)$ of the
subsequence $f_N(t)$ are such that
\begin{equation*}
f(t) \,-\, F(x_0) \,=\, \int_0^t \lambda\, f(r)\, dr \;,\quad 0\le
t\le T \;,
\end{equation*}
which yields uniqueness of limit points.
\qed

\smallskip\noindent{\bf Proof of Theorem \ref{s00}.} Recall that
$T^{\rm mix}_{\mb r, x}$, $x\in S$, stands for the mixing time of the
reflected process $\eta^{\mb r, x} (t)$.  We prove that the
one-dimensional distributions converge. The extension to higher
dimensional distributions is straightforward.  In view of Theorem
\ref{s02}, it is enough to show that for each function $F:S\to \bb R$,
\begin{equation}
\label{46}
\lim_{N\to\infty} \Big|\, \bb E^{\ms E}_{\nu_N} \big[ F(X^N_{t\theta_N}) \big] \;-\;
\bb E^{\ms E}_{\pi_{x_0}} \big[ F(X^N_{t\theta_N}) \big] \, \Big| \;=\;0 \;. 
\end{equation}
Let $T_N$ be a sequence satisfying the assumptions of the theorem.  We
may write $\bb E^{\ms E}_{\nu_N} [ F(X^N_{t\theta_N}) ]$ as
\begin{equation*}
\bb E^{\ms E}_{\nu_N} \big[ \mb 1\{ H_{\breve{\ms E}^{x_0}_N} > T_N\}\,
F(X^N_{t\theta_N}) \big] \;+\; 
\bb E^{\ms E}_{\nu_N} \big[ \mb 1\{ H_{\breve{\ms E}^{x_0}_N} \le T_N\}\,
F(X^N_{t\theta_N}) \big] \;.
\end{equation*} 
The second term is absolutely bounded by $C_0 \bb P^{\ms E}_{\nu_N} [
H_{\breve{\ms E}^{x_0}_N} \le T_N]$ for some finite constant
$C_0$ independent of $N$ and which may change from line to line. By
hypothesis, this latter probability vanishes as $N\uparrow\infty$. By
the Markov property, the first term in the previous displayed equation
is equal to
\begin{equation*}
\bb E^{\ms E}_{\nu_N} \Big[ \mb 1\{ H_{\breve{\ms E}^{x_0}_N} > T_N\}\,
\bb E^{\ms E}_{\eta (T_N)} \big[ F(X^N_{t\theta_N - T_N})
\big]\, \Big] \;.
\end{equation*} 
On the set $\{ H_{\breve{\ms E}^{x_0}_N} > T_N\}$ we may
couple the trace process with the reflected process in such a way that
$\eta^{\ms E} (t) = \eta^{\mb r, x_0}(t)$ for $t\le T_N$. The
previous expectation is thus equal to
\begin{equation*}
\bb E^{\ms E}_{\nu_N} \Big[ \,
\bb E^{\ms E}_{\eta^{\mb r, x_0} (T_N)} \big[ F(X^N_{t\theta_N - T_N})
\big]\, \Big] \;-\;
\bb E^{\ms E}_{\nu_N} \Big[ \mb 1\{ H_{\breve{\ms E}^{x_0}_N} \le T_N\}\,
\bb E^{\ms E}_{\eta^{\mb r, x_0} (T_N)} \big[ F(X^N_{t\theta_N - T_N})
\big]\, \Big]. 
\end{equation*}
As before, the second term vanishes as $N\uparrow\infty$. The first
expectation is equal to
\begin{equation*}
\bb E^{\ms E}_{\pi_{x_0}} \big[ F(X^N_{t\theta_N - T_N}) \big]
\;+\; R_N(t)\;,
\end{equation*}
where $R_N(t)$ is absolutely bounded by $C_0 \Vert \nu_N S^{\mb r, x_0}
(T_N) - \pi_{x_0} \Vert_{\rm TV}$. In this formula, $\Vert \mu
- \nu \Vert_{\rm TV}$ stands for the total variation distance between
$\mu$ and $\nu$ and $S^{\mb r, x} (t)$ represents the semi-group of
the reflected process. By definition of the mixing time, this last
expression is less than or equal to $(1/2)^{(T_N/T^{\rm mix}_{\mb
    r, x})}$, which vanishes as $N\uparrow\infty$ by assumption.

At this point we repeat the same argument with the measure $\nu_N$
replaced by the local equilibrium $\pi_{x_0}$. To estimate $\bb P^{\ms
  E}_{\pi_{x_0}} [ H_{\breve{\ms E}^{x_0}_N} \le T_N]$ we write this
expression as
\begin{equation}
\label{52}
\sum_{\eta\in \ms E^{x_0}_N} \big\{ 
\pi_{x_0}(\eta) - \pi^*_{x_0}(\eta) \big\}\,
\bb P^{\ms E}_{\eta} \big[ H_{\breve{\ms E}^{x_0}_N} \le T_N
\big] \;+\; \bb P^{\ms E}_{\pi^*_{x_0}} \big[ H_{\breve{\ms E}^{x_0}_N} 
\le T_N \big] \;,  
\end{equation}
where $\pi^*_{x_0}$ is the quasi-stationary measure associated to the
trace process $\eta^{\ms E}(t)$ killed when it hits $\breve{\ms
  E}^{x_0}_N$.  The first term is less than or equal to
\begin{equation*}
\sum_{\eta\in \ms E^{x_0}_N} \pi_{x_0}(\eta) \, \Big|
\frac{\pi^*_{x_0}(\eta)}{\pi_{x_0}(\eta)} - 1  \Big| \;\le\;
\Big\{ \sum_{\eta\in \ms E^{x_0}_N} \pi_{x_0}(\eta) \, \Big(
\frac{\pi^*_{x_0}(\eta)}{\pi_{x_0}(\eta)} - 1  \Big)^2 \Big\}^{1/2}\;.
\end{equation*}
By Proposition 2.1, (17) and Lemma 2.2 in \cite{bg11}, the expression
inside the square root on the right hand side of the previous formula
is bounded by $\varepsilon_{x_0} /[1-\varepsilon_{x_0}]$, where
$\varepsilon_{x_0} = E_{\pi_{x_0}}[R^{\ms E}(\eta, \breve{\ms
  E}^{x_0}_N)]/\mf g_{\mb r, x_0}$. By \eqref{43}, $\varepsilon_{x_0}
\le C_0 (\theta_N \mf g_{\mb r, x_0})^{-1}$ for some finite constant
$C_0$ and by hypothesis, $\theta_N^{-1} \ll \mf g_{\mb r, x_0}$ . This
shows that the first term in \eqref{52} vanishes as $N\uparrow\infty$.

Finally, since $\pi^*_{x_0}$ is the quasi-stationary state, under $\bb
P_{\pi^*_x}$, the hitting time of $\breve{\ms E}^{x}_N$,
$H_{\breve{\ms E}^{x}}$, has an exponential distribution whose
parameter we denote by $\phi^*_x$.  By \cite[Lemma 2.2]{bg11},
$\phi^*_x$ is bounded by $E_{\pi_x}[R^{\ms E}(\eta, \breve{\ms
  E}^{x})] \le C_0/\theta_N$, for some finite constant $C_0$. Hence,
\begin{equation*}
\bb P^{\ms E}_{\pi^*_x} \big[ H_{\breve{\ms E}^{x}} \le T_N
\big] \;=\; 1 - e^{-\phi^*_x T_N} \;\le\;
1 - e^{- C_0 (T_N/\theta_N)}\;,
\end{equation*}
an expression which vanishes as $N\uparrow\infty$.
\qed

\smallskip\noindent{\bf Proof of Lemma \ref{s09}.}  Let
$\nu_N$ be a sequence of probability measures satisfying \eqref{32}.
By Schwarz inequality, the square of the expectation appearing in the
statement of the lemma is bounded above by
\begin{equation*}
\frac 1{\pi (\ms E_N)} \, E_{\pi_{\ms E}} \Big[ 
\Big( \frac{d\nu_N}{d\pi_{\ms E}} \Big)^2 \Big] \,
\bb E_{\pi}
\Big[ \Big(\int_0^t {\bs 1}\{\eta (s\theta_N) \in \Delta_N\}
\, ds \Big)^2\Big]
\end{equation*}
By assumption \eqref{32}, the first expectation is bounded by $C_0
\min_{x\in S} \ms M_x^{-1}$. Since $\ms M_x \ge \min_{y} \pi_{\ms
  E}(\ms E^y_N)$, $\min_{x\in S} \ms M_x^{-1} \le \max _{y\in S}
\pi_{\ms E}(\ms E^y_N)^{-1}$. On the other hand, by Schwarz
inequality, the second expectation is less than or equal to
\begin{equation*}
t \, \bb E_{\pi}
\Big[ \int_0^t {\bs 1}\{\eta (s\theta_N) \in \Delta_N\}
\, ds \, \Big] \;=\; t^2 \pi (\Delta_N)\;,
\end{equation*}
which concludes the proof. \qed

\smallskip\noindent{\bf Proof of Lemma \ref{s25}.}  The proof of this
result is similar to the previous one with obvious
modifications. Consider a sequence of initial states $\eta^N$ in $\ms
D^x_N$. By the Markov property, the expectation appearing in
\eqref{58} with $\nu_N = \delta_{\eta^N}$ is bounded above by
\begin{equation*}
\begin{split}
& \bb E_{\eta^N} \Big[  \mb 1\big\{ H_{(\ms F^{x}_N)^c} >
T_N \}\, \bb E_{\eta (T_N)} 
\Big[\int_0^t {\bs 1}\{\eta (s\theta_N) \in \Delta_N\}
\, ds \,\Big]\, \Big] \\
& \quad + \; T_N/\theta_N \;+ \;  t \, 
\bb P_{\eta^N} \big[ H_{(\ms F^{x}_N)^c} \le T_N \;\big] \;,
\end{split}
\end{equation*}
where we replaced $t-T_N$ by $t$ in the time integral.  By assumption,
the second and the third term vanish as $N\uparrow\infty$. On the set
$\{ H_{(\ms F^{x}_N)^c} > T_N \}$ we may replace $\eta (T_N)$ by
$\eta^{\mb r, \ms F^{x}} (T_N)$, where $\eta^{\mb r, \ms F^{x}} (t)$
stands for the process $\eta (t)$ reflected at $\ms F^{x}_N$. After
this replacement, we may remove the indicator and estimate the
expectation by
\begin{equation*}
t\, \Vert \delta_{\eta^N} S^{\mb r, \ms F^{x}} (T_N) - \pi_{\ms F^x} 
\Vert_{\rm  TV} \;+\; \bb E_{\pi_{\ms F^x}} \Big[  
\int_0^t {\bs 1}\{\eta (s\theta_N) \in \Delta_N\}
\, ds \,\Big]\;,
\end{equation*}
where $S^{\mb r, \ms F^{x}} (t)$ represents the semigroup of the
reflected process $\eta^{\mb r, \ms F^{x}} (t)$ and $\pi_{\ms F^x}$
the measure $\pi$ conditioned to $\ms F^{x}_N$.  The first term
vanishes by definition of $T_N$, while the second one is bounded by $t
\pi(\Delta_N)/\pi(\ms F^{x}_N)$, which vanishes in view of condition
({\bf L3}).  \qed

\smallskip\noindent{\bf Proof of Proposition \ref{s16}.}  The proof of
this proposition relies on a comparison between the original process
and the trace process presented below in equations \eqref{55} and
\eqref{54}.  Let $\{T_{\ms E} (t) \,|\, t\ge 0\}$ be the time spent on
the set $\ms E_N$ by the process $\eta (s)$ in the time interval
$[0,t]$,
\begin{equation*}
T_{\ms E} (t) \;=\; \int_0^t \mb 1\{\eta (s)\in \ms E_N\}\, ds\;.
\end{equation*}
Denote by $S_{\ms E} (t)$ the generalized inverse of $T_{\ms E} (t)$,
$S_{\ms E} (t) = \sup \{s\ge 0 \,|\, T_{\ms E} (s) \le t\}$, and
recall that the trace process is defined as $\eta^{\ms E}(t) = \eta
(S_{\ms E} (t))$.

By definition of the trace process, for every $t\ge 0$,
\begin{equation}
\label{55}
\int_0^t \mb 1\{ \eta (s) \in \ms E^{x}_N \} \, ds 
\;\le\;
\int_0^t \mb 1\{ \eta^{\ms E} (s) \in \ms E^{x}_N \} \, ds \;.
\end{equation}
On the other hand,
\begin{equation*}
\int_0^{t} \mb 1\{ \eta^{\ms E} (r) \in \ms E^{x}_N \} \, dr
\;=\; \int_0^{t} \mb 1\{ \eta (S_{\ms E} (r)) \in \ms E^{x}_N \} \, dr\;.
\end{equation*}
By a change of variables, the previous integral is equal to
\begin{equation*}
\int_0^{S_{\ms E} (t)} \mb 1\{ \eta (r) \in \ms E^{x}_N \} \, dr\;.
\end{equation*}
Let $T_\Delta (t)$, $t\ge 0$, be the time spent by the process
$\eta(s)$ on the set $\Delta_N$ in the time interval $[0,t]$,
$T_\Delta (t) = \int_0^{t} \mb 1\{\eta(s) \in \Delta_N\} \, ds$.
Since $T_{\ms E} (t) + T_\Delta (t) = t$, on the set $T_\Delta (t_0)
< \delta$ and for $t\le t_0 - \delta$, $T_{\ms E} (t + \delta) > t$,
so that $S_{\ms E} (t) \le t+\delta$.  Putting together all previous
estimates we get that on the set $T_\Delta (t_0) < \delta$ and for
$t\le t_0 - \delta$,
\begin{equation}
\label{54}
\int_0^t \mb 1\{ \eta^{\ms E} (s) \in \ms E^{x}_N \} \, ds 
\;\le\;
\int_0^{t+\delta} \mb 1\{ \eta (s) \in \ms E^{x}_N \} \, ds\;. 
\end{equation}

We turn now to the proof of the proposition. We may rewrite the time
integral appearing on the left hand side of \eqref{56} as
\begin{equation}
\label{57}
\bb E_{\nu_N} \Big[ \int_0^t  \mb 1\{ \eta( r\theta_N) \in \ms
E^{x}_N \} \, dr \Big]\;.
\end{equation}
By \eqref{55}, this expectation is bounded above by 
\begin{equation*}
\bb E_{\nu_N}
\Big[ \int_0^t \mb 1\{ \eta^{\ms E}( r\theta_N) \in \ms E^{x}_N \} \,
dr \Big] \;=\; \bb E_{\nu_N}
\Big[ \int_0^t \mb 1\{ X^N_{r\theta_N} =x \} \, dr \Big]\;.
\end{equation*}
By Theorem \ref{s00}, the right hand side converges as
$N\uparrow\infty$ to the right hand side of \eqref{56}.

Fix $\delta >0$. The expectation \eqref{57} is bounded below by
\begin{equation*}
\bb E_{\nu_N} \Big[ \mb 1\{ T_\Delta (t\theta_N) < \delta \theta_N\} \,
\int_0^{t} \mb 1\{ \eta( r\theta_N) \in \ms
E^{x}_N \} \, dr \Big]\;.
\end{equation*}
By \eqref{54}, this expression is bounded below by
\begin{equation*}
\begin{split}
& \bb E_{\nu_N} \Big[ \mb 1\{ T_\Delta (t\theta_N) < \delta \theta_N\} \,
\int_0^{t-\delta} \mb 1\{ \eta^{\ms E}( r\theta_N) \in \ms
E^{x}_N \} \, dr \Big] \\
&\qquad \ge\; 
\bb E_{\nu_N} \Big[ 
\int_0^{t-\delta} \mb 1\{ \eta^{\ms E}( r\theta_N) \in \ms
E^{x}_N \} \, dr \Big] \;-\; t \, \bb P_{\nu_N} \big[ T_\Delta
(t\theta_N) \ge \delta \theta_N \big] \;.
\end{split}
\end{equation*}
By \eqref{58}, the second term vanishes as $N\uparrow\infty$, while by
Theorem \ref{s00} the second one converges to the right hand side of
\eqref{56} as $N\uparrow\infty$ and then $\delta\downarrow 0$.
\qed

\smallskip\noindent{\bf The jump rates.}  Recall the definition
\eqref{31} of the rates $r_N(x,y)$. For all $x\in S$,
\begin{equation}
\label{50}
\sum_{y\not = x} r_N(x,y) \;\le\; E_{\pi_x} 
\big[ R^{\ms E}(\eta, \breve{\ms E}^{x}_N) \big] \;.
\end{equation}
Indeed, by \eqref{28} and by the Dirichlet principle,
\begin{equation*}
\pi_\star(\ms E^{x}_N) \sum_{y\not = x} r_N(x,y) \;=\; 
\Cap_\star (\ms E^{\star, x}_N, \breve{\ms E}^{\star, x}_N) \;=\;
\inf_f \< (-L_\star f), f\>_{\pi_\star} \;,
\end{equation*}
where the infimum is carried over all functions $f:\ms E_N \cup \ms
E^\star_N \to \bb R$ equal to $1$ on $\ms E^{\star, x}_N$ and equal to
$0$ on $\breve{\ms E}^{\star, x}_N$. Taking $f= \mb 1\{\ms E^{x}_N
\cup \ms E^{\star, x}_N\}$ and computing the Dirichlet form of this
function we get \eqref{50}. 

\section{On assumptions  (L4) and (L4U)}
\label{sec6}

We present in this section two estimates of $\bb P^{\ms E}_{\eta} [
H_{\breve{\ms E}^{x}_N} \le T_N]$.  We start with a bound of this
probability in terms of an equilibrium potential. Denote by
$W^\star_{x, \gamma}$, $x\in S$, $\gamma>0$, the equilibrium potential
between $\breve{\ms E}^{x}_N \cup \breve{\ms E}^{\star, x}_N$ and $\ms
E^{\star, x}_N$ for the $\gamma$-enlargement of the trace process
$\eta^{\ms E}(t)$:
\begin{equation*}
W^\star_{x, \gamma} (\eta) \;=\; \bb P^{\star,\gamma}_\eta
\big[ H_{\breve{\ms E}^{x}_N \cup \breve{\ms E}^{\star, x}_N} <
H_{\ms E^{\star, x}_N} \big]
\;=\; \bb P^{\star,\gamma}_\eta
\big[ H_{\breve{\ms E}^{x}_N} < H_{\ms E^{\star, x}_N} \big]\;.
\end{equation*}

\begin{lemma}
\label{s11}
Fix $x\in S$. Then, for all $\eta\in\ms E^x_N$, $\gamma>0$ and $A>0$,
\begin{equation*}
\begin{split}
& \bb P^{\ms E}_{\eta} \big[ H_{\breve{\ms E}^{x}_N} \le \gamma^{-1} \big]
\;\le\; e \, W^\star_{x, \gamma} (\eta) \; , \\
& \qquad W^\star_{x, \gamma} (\eta) \;-\;
\frac {e^{-A}} {1- e^{-A}} \;\le\; 
\bb P^{\ms E}_{\eta} \big[ H_{\breve{\ms E}^{x}_N} \le A \gamma^{-1} \big]
\;.
\end{split}
\end{equation*}
\end{lemma}

\begin{proof}
Fix $x\in S$. By definition of the equilibrium potential,
\begin{equation*}
\begin{cases}
(L_{\star} W^\star_{x, \gamma}) (\eta) = 0 & 
\eta\in \ms E^x_N\;,\\
W^\star_{x, \gamma} (\eta) = 1 & \eta\in \breve{\ms E}^{x}_N 
\cup \breve{\ms E}^{\star, x}_N \;,\\
W^\star_{x, \gamma} (\eta) = 0 & \eta\in \ms E^{\star, x}_N  \;.
\end{cases}
\end{equation*}
By definition of the generator $L_{\star}$ and since the equilibrium
potential $W^\star_{x, \gamma}$ vanishes on the set $\ms E^{\star, x}_N$, 
on the set $\ms E^x_N$, we have that
\begin{equation*}
(L_{\ms E} W^\star_{x, \gamma}) (\eta) =  \gamma W^\star_{x, \gamma} (\eta) 
\;,\quad \eta\in \ms E^x_N\;.
\end{equation*}
Since $W^\star_{x, \gamma}$ is equal to $1$ on the set $\breve{\ms
  E}^{x}_N$, we conclude that
\begin{equation*}
W^\star_{x, \gamma} (\eta) \;=\; \bb E^{\ms E}_{\eta} 
[ \exp\{-\gamma H_{\breve{\ms E}^{x}_N}\} ] \;, \quad
\eta\in \ms E_N\;.
\end{equation*}

On the other hand, by Tchebychev inequality and by the previous
identity,
\begin{equation*}
\bb P^{\ms E}_{\eta} \big[ H_{\breve{\ms E}^{x}_N} \le \gamma^{-1}
\big] \;=\; \bb P^{\ms E}_{\eta} \big[ e^{-\gamma 
H_{\breve{\ms E}^{x}_N}} \ge e^{-1} \big]
\;\le\; e\,
\bb E^{\ms E}_{\eta} \big[ e^{-\gamma H_{\breve{\ms E}^{x}_N}}
\big] \;=\; e\, W^\star_{x, \gamma} (\eta) \;.
\end{equation*}

Conversely, fix $A>0$ and let $T(\gamma)$ be an exponential time of
parameter $\gamma$ independent of the trace process $\eta^{\ms
  E}(s)$. It is clear that for $\eta\in \ms E^{x}_N$,
\begin{equation*}
W^\star_{x, \gamma} (\eta) \;=\;
\bb P^{\star, \gamma}_{\eta} \big[ H_{\breve{\ms E}^{x}_N} <  
H_{\ms E^{\star, x}_N} \big ] \;=\;  
\bb P^{\ms E}_{\eta} \big[ H_{\breve{\ms E}^{x}_N} <  
T(\gamma) \big]\;.
\end{equation*}
By definition of $T(\gamma)$, the last probability is equal to
\begin{equation*}
\int_0^\infty \bb P^{\ms E}_{\eta} \big[ H_{\breve{\ms E}^{x}_N} <  
t \big]\, \gamma e^{-\gamma t} \, dt \;\le\;
\bb P^{\ms E}_{\eta} \big[ H_{\breve{\ms E}^{x}_N} \le A \gamma^{-1}
\big] (1- e^{-A}) \;+\; e^{-A} \;.
\end{equation*}
An elementary computation permits to conclude the proof of the lemma.
\end{proof}

The second assertion of the previous lemma shows that we do not lose
much in the first one.

\begin{corollary}
\label{s19}
Let $\nu_N$ be a probability measure concentrated on the set $\ms
E^{x}_N$. Then, for all $\gamma>0$,
\begin{equation*}
\bb P^{\ms E}_{\nu_N} \big[ H_{\breve{\ms E}^{x}_N} \le \gamma^{-1}
\big]^2 \;\le\; \frac{2\,e^2}{\gamma}\,
E_{\pi_{\ms E}} \Big[ \Big( \frac{\nu_N}{\pi_{\ms E}} \Big)^2 \Big]\,
\Cap_{\star} (\ms E^{\star, x}_N , \breve{\ms E}^{x}_N)\;.
\end{equation*}
\end{corollary}

\begin{proof}
Recall that we denote by $\eta^\star$ the copy of the state $\eta$. By
definition of the enlarged process and by Schwarz inequality,
\begin{equation*}
\begin{split}
& \bb P^{\star,\gamma}_{\nu_N}
\big[ H_{\breve{\ms E}^{x}_N} < H_{\ms E^{\star, x}_N}  \big] \;=\;
\sum_{\eta \in \ms E^{x}_N} \nu_N(\eta) \, \bb P^{\star,\gamma}_{\eta^\star}
\big[ H_{\breve{\ms E}^{x}_N} < H_{\ms E^{\star, x}_N}^+  \big] \\
&\qquad \;\le\;
\Big\{ E_{\pi_{\ms E}} \Big[ \Big( \frac{\nu_N}{\pi_{\ms E}} \Big)^2 \Big] 
\,  \sum_{\eta \in \ms E^{x}_N} \pi_{\ms E}(\eta)
\, \bb P^{\star,\gamma}_{\eta^\star}
\big[ H_{\breve{\ms E}^{x}_N} < H_{\ms E^{\star, x}_N}^+  \big] \,
\Big\}^{1/2} \;.
\end{split}
\end{equation*}
In the previous sum we may replace $\pi_{\ms E}(\eta)$ by $2\,
\pi_\star (\eta^\star)$. After the replacement, the sum becomes $2
\gamma^{-1} \Cap_{\star} (\ms E^{\star, x}_N , \breve{\ms
  E}^{x}_N)$. This estimate together with Lemma \ref{s11} concludes
the proof of the corollary.
\end{proof}

\smallskip\noindent{\bf Comments on assumption (L4U).} We present in
this subsection two strategies to prove that the equilibrium potential
$W^\star_{x, \gamma} (\eta)$ vanishes. We apply the first technique in
Example A of Section \ref{sec5}.

\smallskip\noindent\emph{A. Monotonicity.}
On the one hand, it is always possible to couple two trace processes
$\eta^{\ms E}(t)$ starting from different initial states in such a way
that both reach the set $\ms E^{\star}_N$ at the same time.  Assume
that the equilibrium potential $W^\star_{x, \gamma}$ satisfy some
property $\ms P$.  For example, if the state space $\ms E_N$ is
partially ordered and if the process $\eta^{\ms E}(t)$ is monotone,
the equilibrium potential might be monotone. By the Dirichlet
principle,
\begin{equation*}
\Cap_{\star} (\ms E^{\star, x}_N, \breve{\ms E}^{x}_N \cup \breve{\ms
  E}^{\star, x}_N) \;=\;
\< W^\star_{x, \gamma} , (-L_\star) W^\star_{x, \gamma}
\>_{\pi_\star} \;=\; \inf_f \, \< f, (-L_\star) f\>_{\pi_\star} \;,
\end{equation*}
where the supremum is carried over all functions $f$ vanishing at $\ms
E^{\star, x}_N$, equal to $1$ on $\breve{\ms E}^{x}_N \cup \breve{\ms
  E}^{\star, x}_N$ and satisfying condition $\ms P$. Fix a
configuration $\eta\in \ms E^{x}_N$ and denote by $R_N(\varepsilon)$,
$\varepsilon >0$, the right hand side of the previous formula when we
impose the further restriction that $f(\eta) \ge \varepsilon$.

To prove that $W^\star_{x, \gamma} (\eta)$ vanishes as
$N\uparrow\infty$, it is enough to show that for every $\varepsilon
>0$, $\Cap_{\star} (\ms E^{\star, x}_N, \breve{\ms E}^{x}_N \cup
\breve{\ms E}^{\star, x}_N) \ll R_N (\varepsilon)$. Indeed, suppose by
contradiction that $W^\star_{x, \gamma} (\eta)$ does not vanish as
$N\uparrow\infty$. There exists in this case $\varepsilon >0$ and a
subsequence $N_j$, still denoted by $N$, for which $W^\star_{x,
  \gamma} (\eta) \ge \varepsilon$ for all $N$. Therefore,
\begin{equation*}
R_N (\varepsilon) \;\le\; \< W^\star_{x, \gamma} , (-L_\star) W^\star_{x, \gamma}
\>_{\pi_\star} \;=\; \Cap_{\star} (\ms E^{\star, x}_N, 
\breve{\ms E}^{x}_N \cup \breve{\ms E}^{\star, x}_N)\;, 
\end{equation*}
proving our claim.

\smallskip\noindent\emph{B. Capacities.}
To present the second form of estimating the equilibrium potential, we
start with a general result which expresses the equilibrium potential
as a ratio between capacities. Consider a reversible Markov chain
$\eta (t)$ on some countable state space $E$. Denote by $\mb P_\xi$,
$\xi\in E$, the probability measure on the path space $D(\bb R_+, E)$
induced by the Markov process $\eta(t)$ starting from $\xi$, and by
$\Cap (A,B)$ the capacity between two disjoint subsets, $A$, $B$, of
$E$. Next result was communicated to us by A. Teixeira.

\begin{lemma}
\label{s20}
Let $A$, $B$ be two disjoint subsets of $E$, $A\cap B =
\varnothing$, and let $\eta\not\in A\cup B$. Then,
\begin{equation*}
\begin{split}
\mb P_{\eta}
\big[ H_B < H_A  \big] \; & = \;
\frac{\Cap (\eta, A\cup B) + \Cap (B, A\cup \{\eta\}) -
\Cap (A, B\cup \{\eta\})}
{2\, \Cap (\eta, A\cup B)} \\
&\le\; 
\frac{\Cap (\eta, B)} {\Cap (\eta, A\cup B)}
\;\cdot
\end{split}
\end{equation*}
\end{lemma}

\begin{proof}
Denote by $\eta^{T}(t)$ the trace of the process $\eta(t)$ on the set
$A \cup B \cup \{\eta\}$, and by $\mb P^{T}_{\eta}$ the distribution
of this Markov process starting from $\eta$. Clearly,
\begin{equation*}
\mb P_{\eta} \big[ H_{B} < H_{A}  \big] \;=\;
\mb P^{T}_{\eta} \big[ H_{B} < H_{A}  \big]
\; = \; \frac{R_T(\eta, B)} {R_T(\eta, A) + R_T(\eta, B)} \;,
\end{equation*}
if $R_T(\zeta, \xi)$ represents the jump rates of the trace process
$\eta^{T}(t)$. Denote by $\mu_T(\,\cdot\,)$ the stationary measure of
the process $\eta^{T}(t)$.  Multiplying the numerator and the
denominator of the former ratio by $\mu_T(\eta)$, in view of
\cite[Lemma 6.8]{bl2}, the ratio becomes
\begin{equation*}
\begin{split}
& \frac{\Cap_{T}(\eta, A\cup B) + \Cap_{T}(B, A\cup \{\eta\}) -
\Cap_{T}(A, B\cup \{\eta\})}
{2\, \Cap_{T}(\eta, A\cup B)} \\
&\quad \;=\;
\frac{\Cap (\eta, A\cup B) + \Cap (B, A\cup \{\eta\}) -
\Cap (A, B\cup \{\eta\})}
{2\, \Cap (\eta, A\cup B)} \;,    
\end{split}
\end{equation*}
where we used \cite[Lemma 6.9]{bl2} in the last equality. This proves
the identity. To derive the inequality, denote by $\lambda_T(\cdot)$
the holding rates of the trace process $\eta^T(t)$ and observe that
$\Cap_T (\eta, B) = \mu_T(\eta)\, \lambda_T(\eta)\, \mb P^T_\eta [ H_B
< H^+_\eta] \ge \mu_T(\eta)\, R_T(\eta, B)$.
\end{proof}

In some cases the estimate presented in the previous lemma has no
content. On the one hand,
\begin{equation*}
\Cap_T (\eta, B) \;=\;  \mu_T(\eta)\, R_T(\eta, B) 
\;+\; \mu_T(\eta)\, \lambda_T(\eta)\, \mb P^T_\eta
\big[ H_A< H_B < H^+_\eta \big]\;.
\end{equation*}
The second term on the right hand side is the expression we added to
the numerator to transform the identity presented in Lemma \ref{s20}
into an inequality. On the other hand, since $H_A \wedge
H_B < H^+_\eta$ $\mb P^T_\eta$-a.s.,
\begin{equation*}
\mb P^T_\eta \big [ H_A< H_B < H^+_\eta \big] \;=\; 
\mb E^T_\eta \Big [ \mb 1\{ H_A< H_B\} \, \mb P^T_{\eta^T(H_A)} \big [
H_B < H_\eta \big]\; \Big]\;.
\end{equation*}
and
\begin{equation*}
\begin{split}
& \mu_T(\eta)\, R_T(\eta, B) \;+\; \mu_T(\eta)\, \lambda_T(\eta)\, 
\mb P^T_\eta \big[ H_A< H_B \big] \\
&\qquad \;=\; 
\mu_T(\eta)\, R_T(\eta, B) \;+\; \mu_T(\eta)\, R_T(\eta, A)
\;=\; \Cap_T (\eta, A\cup B)\;,
\end{split}
\end{equation*}
which is the expression which appears in the denominator in the proof
of the lemma.  Therefore, the statement of the lemma may have some
interest only if $\mb P^T_{\eta^T(H_A)} \big [ H_B < H_\eta \big] =
\mb P_{\eta^T(H_A)} \big [ H_B < H_\eta \big]$ is negligible, i.e., if
the process starting from $A$ reaches $B$ before $\eta$ with a
vanishing probability. \smallskip

We apply Lemma \ref{s20} to our context to obtain a bound on $\bb
P^{\ms E}_{\eta} [ H_{\breve{\ms E}^{x}_N} \le \gamma^{-1} ]$. For
$\gamma >0$, consider the Markov process $\{\eta^{N,\star}(t) : t\ge
0\}$ on $E_N \cup \ms E^\star_N$ whose jump rates $R_{N,\star} (\eta,
\xi) = R^{\gamma}_{N,\star} (\eta, \xi)$ are given by
\begin{equation*}
R_{N,\star}(\eta,\xi) \;=\; 
\begin{cases}
R_N(\eta,\xi) & \text{if $\eta$ and $\xi\in E_N$,} \\
\gamma & \text{if $\eta\in \ms E^{\star}_N \cup
\ms E_N $ and if [$\xi = \eta^\star$ 
or $\eta = \xi^\star$],} \\
0 & \text{otherwise.} 
\end{cases}
\end{equation*}
Note that the proces $\eta^{\star}(t)$ is the trace of the process
$\eta^{N,\star}(t)$ on $\ms E^{\star}_N \cup \ms E_N$. Denote by
$\Cap_{N,\star}$ the capacity associated to the process
$\eta^{N,\star}(t)$.  Next result provides a bound for condition ({\bf
  L4U}) in terms of capacities which can be estimated through the
Dirichlet and the Thomson principles.

\begin{corollary}
\label{s21}
For every $x\in S$, $\eta\in \ms E^{x}_N$ and $\gamma>0$,
\begin{equation*}
\bb P^{\ms E}_{\eta} \big [ H_{\breve{\ms E}^{x}_N} \le \gamma^{-1}
\big ] \;\le\;
\frac{e\, \Cap_{N}(\eta, \breve{\ms E}^{x}_N)}
{2\, \Cap_{N,\star}(\eta, \ms E^{\star, x}_N)} \;\cdot
\end{equation*}
\end{corollary}

\begin{proof}
By Lemma \ref{s11} and by Lemma \ref{s20},
\begin{equation*}
\bb P^{\ms E}_{\eta} \big [ H_{\breve{\ms E}^{x}_N} \le \gamma^{-1}
\big ] \;\le\; e\, \bb P^{\star,\gamma}_{\eta}
\big[ H_{\breve{\ms E}^{x}_N} < H_{\ms E^{\star, x}_N}  \big]
\;\le\; \frac{e\, \Cap_{\star}(\eta, \breve{\ms E}^{x}_N)}
{\Cap_{\star}(\eta, \ms E^{\star, x}_N)} \;\cdot
\end{equation*}
It is clear from the Dirichlet principle and from the definition of
the enlarged process that $\Cap_{\star}(\eta, \breve{\ms E}^{x}_N) =
(1/2) \, \Cap_{\ms E}(\eta, \breve{\ms E}^{x}_N)$, where $\Cap_{\ms E}$
stands for the capacity associated to the trace process $\eta^{\ms
  E}(t)$.  By \cite[Lemma 6.9]{bl2}, once more, $\Cap_{\ms E}(\eta,
\breve{\ms E}^{x}_N) = \pi (\ms E_N)^{-1} \Cap_N (\eta, \breve{\ms
  E}^{x}_N)$ and $\Cap_{\star}(\eta, \ms E^{\star, x}_N) = \pi (\ms
E_N)^{-1} \Cap_{N, \star}(\eta, \ms E^{\star, x}_N)$.  This concludes
the proof of the lemma.
\end{proof}

\section{Tightness}
\label{sec2}

We prove in this section tightness of the process ${\mb X}^N_t$.  By
Aldous criterion (see Theorem 16.10 in \cite{b}) we just need to show
that for every $\epsilon>0$ and $T>0$
\begin{equation}
\label{62}
\lim_{\delta\downarrow 0}\lim_{N\to\infty}\sup_{a\le\delta}
\sup_{\tau\in\mf T_T}\bb P^{\ms E}_{\nu_N} \big[ \;
| {\mb X}^N_{\tau + a} - {\mb X}^N_{\tau}| > \epsilon \;\big] \;=\; 0\;,
\end{equation}
where $\mf T_T$ is the set of stopping times bounded by $T$.

In fact, in the present context of a finite state space, we do not
need to consider all stopping times, but just the jump times. More
precisely, the process ${\mb X}^N_t$ is tight provided 
\begin{equation*}
\lim_{\delta\to 0} \limsup_{N\to\infty} \sup_{i\ge 0} \bb P^{\ms
  E}_{\nu_N} \big[ \tau_{i+1} - \tau_i \le \delta \big]\;=\; 0\;,
\end{equation*}
where $\tau_0=0$ and $\tau_i$, $i\ge 1$, represent the jumping times
of the process ${\mb X}^N_t$.

\smallskip\noindent{\bf Proof of Lemma \ref{s15}.}  We will prove that
\eqref{62} holds. Fix $T>0$, $\epsilon >0$ and $\delta >0$.  By
the strong Markov property, for every $0<a\le \delta$ and stopping
time $\tau \le T$,
\begin{equation*}
\begin{split}
& \bb P^{\ms E}_{\nu_N} \big[ \; | {\mb X}^N_{\tau + a} - {\mb X}^N_{\tau}| > \epsilon
\;\big] \;\le\; \bb P^{\ms E}_{\nu_N} \big[ \; \bb P^{\ms E}_{\eta (\tau)}
\big[ \; | {\mb X}^N_{a} - {\mb X}^N_{0}| > \epsilon \;\big] \,\big] \\
& \qquad \le\; \sup_{\eta\in \ms E_N} \bb P^{\ms E}_{\eta} 
\big[ \; | {\mb X}^N_{a} - {\mb X}^N_{0}| > \epsilon \;\big] \;\le\;
\max_{x\in S} \sup_{\eta\in \ms E^x_N} \bb P^{\ms E}_{\eta} 
\big[ H_{\breve{\ms E}^{x}} \le \delta \theta_N \,\big]\;. 
\end{split}
\end{equation*}
To conclude the proof we need to show that the last term vanishes as
$N\uparrow\infty$ and then $\delta\downarrow 0$. The arguments used
are similar to the ones used in the proof of Theorem \ref{s00}.

Let $T_N$ be a sequence satisfying the assumptions ({\bf L4U}). Fix
$x\in S$ and $\eta\in \ms E^x_N$. The probability $\bb P^{\ms
  E}_{\eta} [ H_{\breve{\ms E}^{x}} \le \delta \theta_N ]$ is bounded
above by
\begin{equation}
\label{40}
\bb P^{\ms E}_{\eta} \big[ H_{\breve{\ms E}^{x}} \le T_N
\;\big] \;+\; \bb E^{\ms E}_{\eta} \Big[ \mb 
1 \big \{H_{\breve{\ms E}^{x}} > T_N \big\}
\, \bb P^{\ms E}_{\eta (T_N)} 
\big[ H_{\breve{\ms E}^{x}} \le \delta \theta_N \;\big] \,\Big] \;.
\end{equation}
The first term vanishes in view of assumption ({\bf L4U}).  On the set
$\{H_{\breve{\ms E}^{x}} > T_N\}$, we may couple the process $\eta
(t)$ with the reflected process $\eta^{\mb r, x} (t)$ in a way that
$\eta (t) = \eta^{\mb r, x} (t)$ for $0\le t \le T_N$. In particular,
we may replace in the previous term $\bb P^{\ms E}_{\eta
  (T_N)} [ H_{\breve{\ms E}^{x}} \le \delta \theta_N]$ by $\bb
P^{\ms E}_{\eta^{{\mb r, x}} (T_N)} [ H_{\breve{\ms E}^{x}}
\le \delta \theta_N]$.  After this replacement we may bound the second
term in \eqref{40} by
\begin{equation}
\label{41}
\sum_{\xi \in \ms E^x_N}  \, \big\{ \big( \delta_\eta
S^{\mb r, x}(T_N) \big) (\xi) -  \pi_x(\xi) \big\}\,
\bb P^{\ms E}_{\xi} 
\big[ H_{\breve{\ms E}^{x}} \le \delta \theta_N \big]
\;+\; \bb P^{\ms E}_{\pi_x} 
\big[ H_{\breve{\ms E}^{x}} \le \delta \theta_N \big] \;,
\end{equation}
where $S^{\mb r, x} (t)$ represents the semi-group of the reflected
process.  The first term of this sum is bounded by $\Vert \delta_\eta
S^{\mb r, x} (T_N) - \pi_x \Vert_{\rm TV}$, where $\Vert \mu - \nu
\Vert_{\rm TV}$ stands for the total variation distance between $\mu$
and $\nu$. By definition of the mixing time, this last expression is
less than or equal to $(1/2)^{(T_N/T^{\rm mix}_{\mb r, x})}$, which
vanishes as $N\uparrow\infty$ by definition of the sequence $T_N$.

It remains to estimate the second term in \eqref{41}. It can be
written as
\begin{equation}
\label{44}
\sum_{\eta\in \ms E^x_N} \big\{ 
\pi_x(\eta) - \pi^*_x(\eta) \big\}\,
\bb P^{\ms E}_{\eta} \big[ H_{\breve{\ms E}^{x}} \le \delta \theta_N
\big] \;+\; \bb P^{\ms E}_{\pi^*_x} \big[ H_{\breve{\ms E}^{x}} 
\le \delta \theta_N \big] \;,  
\end{equation}
where $\pi^*_x$ is the quasi-stationary measure associated to the
trace process $\eta^{\ms E}(t)$ killed when it hits $\breve{\ms
  E}^{x}$.  The first term is less than or equal to
\begin{equation*}
\sum_{\eta\in \ms E^x_N} \pi_x(\eta) \, \Big|
\frac{\pi^*_x(\eta)}{\pi_x(\eta)} - 1  \Big| \;\le\;
\Big\{ \sum_{\eta\in \ms E^x_N} \pi_x(\eta) \, \Big(
\frac{\pi^*_x(\eta)}{\pi_x(\eta)} - 1  \Big)^2 \Big\}^{1/2}\;.
\end{equation*}
By Proposition 2.1, (17) and Lemma 2.2 in \cite{bg11}, the expression
inside the square root on the right hand side is bounded by
$\varepsilon_x /[1-\varepsilon_x]$, where $\varepsilon_x =
E_{\pi_x}[R^{\ms E}(\eta, \breve{\ms E}^{x})]/\mf g_{\mb r, x}$.  By
\eqref{51}, $\varepsilon_{x} \le C_0 (\theta_N \mf g_{\mb r, x})^{-1}$
for some finite constant $C_0$ and by hypothesis, $\theta_N^{-1} \ll
\mf g_{\mb r, x}$ . This shows that the first term in \eqref{44}
vanishes as $N\uparrow\infty$.

Finally, since $\pi^*_x$ is the quasi-stationary state, under $\bb
P_{\pi^*_x}$, the hitting time of $\breve{\ms E}^{x}_N$,
$H_{\breve{\ms E}^{x}}$, has an exponential distribution whose
parameter we denote by $\phi^*_x$.  By \cite[Lemma 2.2]{bg11}, $\phi^*_x$
is bounded by $E_{\pi_x}[R^{\ms E}(\eta, \breve{\ms E}^{x})] \le
C_0/\theta_N$, for some finite constant $C_0$. Hence,
\begin{equation*}
\bb P^{\ms E}_{\pi^*_x} \big[ H_{\breve{\ms E}^{x}} \le \delta
\theta_N \big] \;=\; 1 - e^{-\phi^*_x \delta \theta_N} \;\le\;
1 - e^{-C_0 \delta}\;,
\end{equation*}
an expression which vanishes as $\delta\downarrow 0$. This proves
\eqref{62} and concludes the proof of the lemma. \qed
\smallskip

By a version of \cite[Theorem 12.3]{lpw1} for continuous-time
reversible Markov chains, $T^{\rm mix}_{\mb r, x} \le \mf g_{\mb r,
  x}^{-1} \log (4/\min_{\eta\in \ms E^x_N} \pi_x(\eta))$. Hence,
$\max_{x\in S} T^{\rm mix}_{\mb r, x} \ll T_N$ if 
\begin{equation}
\label{45}
\lim_{N\to\infty} \frac {1} {T_N\, \mf g_{\mb r, x}} \, 
\log \frac 1 {\min_{\eta\in \ms E^x_N} \pi_x(\eta)} \;=\; 0\;.
\end{equation}

\section{The spectral gap of the trace process}
\label{sec3}

We prove in this section Proposition \ref{s05}.  We start with an
elementary result which provides an upper bound for the spectral gap
of the trace process in terms of capacities. Recall that $\eta(t)$ is
a positive recurrent, reversible, continuous-time Markov chain on a
countable state space $E_N$, whose embedded discrete-time chain is
also positive recurrent. Let $\ms E_N$ a subset of $E_N$ and denote by
$\mf g_{\ms E}$ the spectral gap of the trace of $\eta (t)$ on $\ms
E_N$.

\begin{lemma}
\label{s22}
We have that
\begin{equation*}
\mf g_{\ms E} \;\le\; \inf_{\ms A\subset \ms E_N} 
\frac{\pi (\ms E_N)\, \Cap (\ms A, \ms B)}
{\pi (\ms A) \, \pi (\ms B)}\;,
\end{equation*}
where $\ms B = \ms E_N\setminus\ms A$.
\end{lemma}

\begin{proof}
Fix a subset $\ms A$ of $\ms E_N$, and let $\ms B = \ms E_N\setminus\ms A$.
By definition,
\begin{equation*}
\mf g_{\ms E} \;=\; \inf_f \frac{\< f, (-L_{\ms E}) f
  \>_{\pi_{\ms E}} }{\text{Var}_{\pi_{\ms E}} (f)}
\;\le\; \frac{\< \mb 1 \{\ms A\}, (-L_{\ms E}) \mb 1 \{\ms A\}
  \>_{\pi_{\ms E}} }{\text{Var}_{\pi_{\ms E}} (\mb 1 \{\ms A\})} \;,
\end{equation*} 
where $\text{Var}_{\pi_{\ms E}} (f)$ stands for the variance of $f$
with respect to the measure $\pi_{\ms E}$. Since $\ms E_N = \ms A \cup
\mc B$, $\mb 1\{\ms A\}$ is the equilibrium potential between $\ms A$
and $\ms B$ so that $\< \mb 1 \{\ms A\}, (-L_{\ms E}) \mb 1 \{\ms A\}
\>_{\pi_{\ms E}} = \Cap_{\ms E} (\ms A, \ms B)$. Hence, by \cite[Lemma
6.9]{bl2},
\begin{equation*}
\mf g_{\ms E} \;\le\; \frac{\Cap_{\ms E} (\ms A, \ms B)}{\pi_{\ms E}
(\ms A) \, \pi_{\ms E} (\ms B)} \;=\;
\frac{\pi (\ms E_N)\, \Cap (\ms A, \ms B)}
{\pi (\ms A) \, \pi (\ms B)}\;\cdot
\end{equation*}
\end{proof}

\noindent{\bf Proof of Proposition \ref{s05}}.
Let $F:\ms E_N\to\bb R$ be a function in $L^2(\pi_{\ms E})$ and denote
by $\hat F : E_N\to\bb R$ the harmonic extension of $F$ to $E_N$,
defined by
\begin{equation*}
\hat F (\eta) \;=\; 
\begin{cases}
F(\eta) & \text{if $\eta\in \ms E_N$,}\\
\bb E_{\eta}[F(\eta(H_{\ms E_N}))] & \text{if $\eta\not\in \ms E_N$.}
\end{cases}
\end{equation*}
We claim that
\begin{equation}
\label{33}
\< (-L_N) \hat F \,,\, \hat F\>_\pi \;=\; \pi(\ms E_N)\,
\< (-L_{\ms E}) F \,,\, F\>_{\pi_{\ms E}} \;.
\end{equation}
Indeed, since $ L_N \hat F =0$ on $\ms E^c_N$ and since $\hat F$ and
$F$ coincide on $\ms E_N$, the Dirichlet form $\< L_N \hat F \,,\,
\hat F\>_\pi$ is equal to
\begin{equation}
\label{34}
\sum_{\eta \in\ms E_N} \pi(\eta)\, F(\eta)\, \sum_{\xi\in E_N}
R_N(\eta,\xi) \{ \hat F(\xi) - F(\eta) \}\;.
\end{equation}
We decompose the previous sum in two expressions, the first one
including all terms for which $\xi$ belongs to $\ms E_N$ and the
second one including all terms for which $\xi$ belongs to
$E_N\setminus \ms E_N$. When $\xi$ belongs to $\ms E_N$, we may
replace $\hat F$ by $F$. The other expression, by definition of $\hat
F$ is equal to
\begin{equation*}
\sum_{\eta \in\ms E_N} \sum_{\xi\not \in \ms E_N} \pi(\eta)\, F(\eta)\, 
R_N(\eta,\xi) \sum_{\zeta\in \ms E_N} \bb P_\xi \big[ H_{\ms E_N} =
H_\zeta \big]  \{ F(\zeta) - F(\eta) \}\;.
\end{equation*}
Since for $\eta\in\ms E_N$,
\begin{equation*}
\bb P_\eta \big[ H^+_{\ms E_N} = H_\zeta \big] \;=\; p_N(\eta,\zeta)
\;+\; \sum_{\xi\not \in \ms E_N} p_N(\eta,\xi)\,
\bb P_\xi \big[ H_{\ms E_N} = H_\zeta \big] \;,
\end{equation*}
and since by \cite[Proposition 6.1]{bl2} $R^{\ms E}(\eta,\zeta) =
\lambda_N(\eta) \bb P_\eta \big[ H^+_{\ms E_N} = H_\zeta \big]$ the
previous sum is equal to
\begin{equation*}
\sum_{\eta \in\ms E_N}  \sum_{\zeta\in \ms E_N} \pi(\eta)\, F(\eta)\, 
\big\{R^{\ms E}(\eta,\zeta)  - R_N(\eta,\zeta) \big\} \,
\{ F(\zeta) - F(\eta) \}\;. 
\end{equation*}
Adding this sum to the first expression in our decomposition of
\eqref{34} as the sum of two terms, we get that the left hand side of
\eqref{33} is equal to
\begin{equation*}
\sum_{\eta \in\ms E_N}  \sum_{\zeta\in \ms E_N} \pi(\eta)\, F(\eta)\, 
R^{\ms E}(\eta,\zeta) \, \{ F(\zeta) - F(\eta) \} \;.   
\end{equation*}
To conclude the proof of Claim \eqref{33}, it remains to recall that
$\pi_{\ms E}(\eta) = \pi(\eta)/\pi(\ms E_N)$.

Fix a function $F:\ms E_N\to\bb R$. We claim that
\begin{equation}
\label{35}
\inf_g \< (-L_N) g \,,\, g\>_\pi\;=\;
\< (-L_N) \hat F \,,\, \hat F\>_\pi \;,
\end{equation}
where the infimum is carried over all functions $g:E_N \to \bb R$
which are equal to $F$ on $\ms E_N$. Indeed, it is simple to show that
any function $f$ which solves the variational problem on the left hand
side of \eqref{35} is harmonic on $\ms E^c_N$ and coincides with $F$
on $\ms E_N$, $L_N f = 0$ on $\ms E_N^c$ and $f=F$ on $\ms E_N$. The
unique solution to this problem is $\hat F$, which proves \eqref{35}.

Fix an eigenfunction $F$ associated to $\mf g_{\ms E}$ such that
$E_{\pi_{\ms E}}[F^2] = 1$, $E_{\pi_{\ms E}}[F] = 0$. By
\eqref{33} we have that
\begin{equation*}
\mf g_{\ms E} \;=\; \< (-L_{\ms E}) F \,,\, F\>_{\pi_{\ms E}}
\;=\; \frac 1{\pi(\ms E_N)} \< (-L_N) \hat F \,,\, \hat F\>_\pi\;.
\end{equation*}
By the spectral gap, the Dirichlet form on the right hand side is
bounded below by $\mf g$ times the variance of $\hat F$. This latter
variance, in view of the definition of $\hat F$ and the properties of
$F$, is equal to
\begin{equation*}
\pi(\ms E_N) \;+\; \sum_{\eta\not \in \ms E_N} \pi(\eta) \hat F(\eta)^2 
- \Big( \sum_{\eta\not \in \ms E_N} \pi(\eta) \hat F(\eta) \Big) ^2
\;\ge\; \pi(\ms E_N)\;.
\end{equation*}
This proves that $\mf g \le \mf g_{\ms E}$.

Fix an eigenfunction $f$ associated to $\mf g$ such that $E_{\pi}[f^2]
= 1$, $E_{\pi}[f] = 0$. Let $F: \ms E_N \to \bb R$ be the restriction
to $\ms E_N$ of $f$: $F(\eta) = f(\eta) \mb 1\{\eta\in \ms E_N\}$. By
definition of $\mf g$, 
\begin{equation*}
\mf g \;=\; \< (-L_N) f \,,\, f\>_{\pi}
\;\ge\; \inf_g \< (-L_N) g \,,\, g\>_{\pi}\;,
\end{equation*}
where the infimum is carried over all functions $g$ which coincide
with $F$ on $\ms E_N$. By \eqref{35}, by \eqref{33} and by definition
of the spectral gap $\mf g_{\ms E}$, the right hand
side of the previous term is equal to 
\begin{equation*}
\< (-L_N) \hat F \,,\, \hat F\>_{\pi} \;=\; \pi(\ms E_N)\,
\< (-L_{\ms E}) F \,,\, F\>_{\pi_{\ms E}} \;\ge\;
\mf g_{\ms E} \, \pi(\ms E_N)\,\big\{ E_{\pi_{\ms E}} [F^2] - 
E_{\pi_{\ms E}} [F]^2 \big\}\;.
\end{equation*}
Since $F=f\mb 1\{\ms E_N\}$, up to this point we proved that
\begin{equation*}
\mf g_{\ms E} \, \pi(\ms E_N)\,\big\{ E_{\pi_{\ms E}} [f^2 1\{\ms E_N\}] - 
E_{\pi_{\ms E}} [f 1\{\ms E_N\}]^2 \big\} \;\le\; \mf g
\end{equation*}
Since the eigenfunction $f$ associated to $\mf g$ is such that
$E_{\pi}[f^2] = 1$, $E_{\pi}[f] = 0$, we may rewrite the previous
inequality as
\begin{equation*}
\mf g_{\ms E} \, \Big\{ 1 - \Big[
E_{\pi} \big[f^2 \mb 1\{\ms E^c_N\}\big] \;+\;
\frac 1{\pi(\ms E_N)} E_{\pi} \big[f \mb 1\{\ms E^c_N\} \big]^2 \Big] \Big\} 
\;\le\; \mf g \;. 
\end{equation*} 
By Schwarz inequality, $E_{\pi} [f \mb 1\{\ms E^c_N\} ]^2 \le E_{\pi}
[f^2 \mb 1\{\ms E^c_N\} ] \pi (\ms E^c_N) $ so that
\begin{equation*}
\mf g_{\ms E} \, \Big\{ 1 - \frac 1{\pi(\ms E_N)}
E_{\pi} \big[f^2 \mb 1\{\ms E^c_N\}\big] \Big\} 
\;\le\; \mf g \;,
\end{equation*}
which proves the proposition. \qed

\section{Proof of Theorem \ref{s07}}
\label{sec4}

We assume in this section that the state space $E_N$ has been divided
in three disjoint sets $\ms E^1_N = \ms A$, $\ms E^2_N = \ms B$ and
$\Delta_N = E_N \setminus \ms E_N$, where $\ms E_N = \ms A \cup \ms
B$.  Recall that $\eta^{\ms E}(t)$ represents the trace of the process
$\eta(t)$ on the set $\ms E_N$ and $\eta^{\star}(t)$ the
$\gamma$-enlargement of the process $\eta^{\ms E}(t)$ to the set $\ms
E_N \cup \ms E^\star_N$, where $\gamma = \gamma_N$ is a sequence of
positive numbers and $\ms E^\star_N = \ms A^\star \cup \ms B^\star$,
$\ms A^\star$, $\ms B^\star$ being copies of the sets $\ms A$, $\ms
B$, respectively. Denote by $\mf g_{\ms A}$, $\mf g_{\ms B}$ the
spectral gap of the process $\eta(t)$ reflected at $\ms A$, $\ms B$,
respectively. 

Let $\widehat{\Cap}_\star (\ms A^\star, \ms B^\star)$ be the
normalized capacity between $\ms A^\star$ and $\ms B^\star$:
\begin{equation*}
\widehat{\Cap}_\star (\ms A^\star, \ms B^\star) \;=\; 
\frac{\Cap_\star (\ms A^\star, \ms B^\star)}
{\pi_{\ms E}(\ms A)\, \pi_{\ms E}(\ms B)}\;\cdot
\end{equation*}
By \cite[Theorem 2.12]{bg11},
\begin{equation} 
\label{36}
\Big( 1 \;-\; \frac{2 \, \widehat{\Cap}_\star (\ms A^\star, \ms B^\star)}
{\gamma} \Big)^2 \;\le\;
\frac{2 \, \widehat{\Cap}_\star (\ms A^\star, \ms B^\star)}{\mf g_{\ms
    E}} \;\le\; 1 \;+\; \frac{\gamma + 2 \, \widehat{\Cap}_\star (\ms
  A^\star, \ms B^\star)}{\min\{\mf g_{\ms A} , \mf g_{\ms B} \}}\;\cdot
\end{equation}
The factor $2$, which is not present in \cite{bg11}, appears because
we consider the capacity with respect to the probability measure
$\pi_\star$, while \cite{bg11} defines the capacity with $\pi_{\ms E}$
as reference measure.

Theorem \ref{s07} is a simple consequence of \eqref{36}. For sake of
completeness, we present a proof of the lower bound of \eqref{36}. Let
$V$ be the equilibrium potential between $\ms A^\star$ and $\ms
B^\star$: $V(\eta) = \bb P^{\star, \gamma}_\eta[H_{\ms A^\star} <
H_{\ms B^\star}]$. We sometimes consider below $V$ as a function on
$\ms E_N$. By definition of the spectral gap,
\begin{equation*}
\mf g_{\ms E} \;\le\; \frac{\< V, (-L_{\ms E}) V
  \>_{\pi_{\ms E}}}{\text{Var}_{\pi_{\ms E}} (V)}\;,
\end{equation*}
where $\text{Var}_{\pi_{\ms E}} (V)$ stands for the variance of $V$
with respect to the measure $\pi_{\ms E}$. We estimate the numerator
and the denominator separately.

Since the capacity between $\ms A^\star$ and $\ms
B^\star$ is equal to the Dirichlet form of the equilibrium potential,
\begin{equation*}
(1/2) \< V, (-L_{\ms E}) V \>_{\pi_{\ms E}} \;\le\; 
\Cap_\star (\ms A^\star, \ms B^\star) \; .
\end{equation*}
A martingale decomposition of the variance of $V$ gives that
\begin{equation*}
\text{Var}_{\pi_{\ms E}} (V) \;\ge \; \pi_{\ms E}(\ms A)\, 
\pi_{\ms E}(\ms B)\, 
\Big( E_{\pi_{\ms A}} [ V_{\ms A}] - E_{\pi_{\ms B}} [ V_{\ms B}]
\Big)^2\;,
\end{equation*}
where $V_{\ms A} = V \mb 1\{\ms A\}$, $V_{\ms B} = V \mb 1\{\ms B\}$.
Since $\Cap_\star (\ms A^\star, \ms B^\star) = \< V,
(-L_\star)V\>_{\pi_\star}$, since $(L_\star V)(\eta^\star) = \gamma
[V(\eta)-1]$, where $\eta^\star$ is the state in $\ms E^\star_N$
corresponding to the state $\eta\in\ms E_N$, and since $\pi_\star
(\eta^\star) = (1/2) \pi_{\ms E}(\eta)$, $2 \Cap_\star (\ms A^\star,
\ms B^\star) = \gamma \pi_{\ms E}(\ms A) - \gamma \sum_{\eta\in \ms A}
\pi_{\ms E}(\eta) V(\eta)$. Therefore,
\begin{equation*}
E_{\pi_{\ms A}} [ V_{\ms A}] \;=\; 1 \;-\; 
\frac{2\, \Cap_\star (\ms A^\star, \ms B^\star)}
{\gamma \, \pi_{\ms E}(\ms A)}\;\cdot
\end{equation*}
Repeating the previous argument with $1-V$ in place of $V$ we obtain
that
\begin{equation*}
E_{\pi_{\ms B}} [ V_{\ms B}] \;=\; 
\frac{2\, \Cap_\star (\ms A^\star, \ms B^\star)}
{\gamma \, \pi_{\ms E}(\ms B)}\;\cdot
\end{equation*}
Putting together the previous estimates, we conclude the proof of the
lower bound of \eqref{36}. \qed

\section{Applications}
\label{sec5}

We present in this section two applications of Theorems \ref{s02} and
\ref{s00}. Both processes do not visit points in the time scale where
tunneling occurs and, therefore, do not satisfy the assumptions of the
theory developed in \cite{bl2,bl7}. Furthermore, these models have
logarithmic energy or entropy barriers, restraining the application of
large deviations methods. On the other hand, both dynamics are
monotone with respect to a partial order, allowing the use of coupling
techniques.  The first model, which has only entropy barriers, was
suggested by A. Gaudilli\`ere to the authors as a model for testing
metastability techniques. We prove for to this model the mixing
conditions introduced in Section \ref{sec0}.E. The second one has been
examined in details in \cite{cmt, clmst}. We apply to this model the
$L^2$-theory presented in Section \ref{sec0}.D.

\subsection{The dog graph \cite{sc1}}
For $N\ge 1$ and $d\ge 2$, let $Q_N = \{0, \dots, N\}^d$ be a
$d$-dimensional cube of length $N$, let $\breve{Q}_N$ be the
reflection of $Q_N$ through the origin, $\breve{Q}_N = \{\eta\in \bb
Z^d : -\eta\in Q_N\}$, and let $V_N = Q_N \cup \breve{Q}_N$. Denote by
$E_N$ the set of edges formed by pairs of nearest-neighbor sites of
$V_N$, $E_N = \{ (\eta,\xi)\in V_N\times V_N: |\eta-\xi|=1\}$.  The
graph $G_N=(V_N,E_N)$ is called the dog graph \cite{sc1}.

Let $\{\eta(t) : t\ge 0\}$ be the continuous-time Markov chain on
$G_N$ which jumps from $\eta$ to $\xi$ at rate $1$ if $(\eta,\xi)\in
E_N$. The uniform measure on $V_N$, denoted by $\pi$, is the unique
stationary state.  Diaconis and Saloff Coste \cite[Example 3.2.5]{sc1}
proved that there exist constants $0<c(d) < C(d)<\infty$ such that for
all $N\ge 1$,
\begin{equation}
\label{63}
\frac{c(2)}{N^2 \,\log N} \;\le\; \mf g  \;\le\; \frac{C(2)}{N^2
  \,\log N} \quad\text{in $d=2$ and}\quad
\frac{c(d)}{N^d} \;\le\; \mf g  \;\le\; \frac{C(d)}{N^d}
\end{equation}
in dimension $d\ge 3$.

Fix a sequence $\alpha_N$, $(\log N)^{-1/2} \ll \alpha_N \ll 1$, and
let $\ms B_N = \{\eta =(\eta_1, \dots, \eta_d) \in V_N : \min_j \eta_j
\ge \alpha_N\, N\}$, $\ms A_N = - \ms B_N = \{\eta\in V_N : -\eta\in
\ms B_N\}$. Denote by $\mf g_{\ms A}$ and $T^{\rm mix}_{\mb r, \ms
  A}$ (resp. $\mf g_{\ms B}$ and $T^{\rm mix}_{\mb r, \ms B}$) the
spectral gap and the mixing time of the continuous-time random walk
$\eta(t)$ reflected at $\ms A_N$ (resp. $\ms B_N$).  It is well known
that there exist finite constants $0<c(d) < C(d)<\infty$ such that for
all $N\ge 1$,
\begin{equation}
\label{70}
\frac{c(d)}{N^2} \;\le\; \mf g_{\ms A}  \;\le\; \frac{C(d)}{N^2} \;,
\quad
c(d)\, N^2 \;\le\; T^{\rm mix}_{\mb r, \ms A} \;\le\; C(d) N^2\;,
\end{equation}
with similar inequalities if $\ms B$ replaces $\ms A$.

\smallskip\noindent\emph{Condition \eqref{43}}. Let $\ms E_N = \ms A_N
\cup \ms B_N$, and recall the notation introduced in Section
\ref{sec0}.  We claim that condition \eqref{43} is fulfilled for
$\theta_N = N^2 \log N$ in dimension $2$ and for $\theta_N = N^d$ in
dimension $d\ge 3$. Indeed, if $\pi_{\ms A}$, $\pi_{\ms B}$, $\pi_{\ms
  E}$ represent the uniform measure $\pi$ conditioned to $\ms A_N$,
$\ms B_N$, $\ms E_N$, respectively, by \eqref{61},
\begin{equation*}
E_{\pi_{\ms A}}[R^{\ms E}(\eta, \ms B)]  \;=\; \frac 1{\pi(\ms A_N)}\, 
\Cap_N (\ms A_N, \ms B_N)\;.
\end{equation*}
By the Dirichlet principle, the capacity is bounded by the Dirichlet
form of any function which vanishes on $\ms A_N$ and is equal to $1$
on $\ms B_N$. In dimension $d\ge 3$ we simply choose the indicator of
the set $Q_N$. In dimension $2$, let $D_k = \{\eta \in Q_N :
\eta_1+\eta_2 = k\}$, $k\ge 0$. Fix $1\le L\le N$ and consider the function
$f_L: Q_N \to \bb R_+$ defined by $f(0) =0$,
\begin{equation}
\label{66}
f_L(\eta) \;=\; \frac 1{\Phi(L)} \sum_{j=1}^k \frac 1j \quad \eta\in
D_k \;, \quad 1\le k\le L\;,
\end{equation}
where $\Phi(L)= \sum_{1\le j\le L} j^{-1}$, and $f_L(\eta) = 1$
otherwise. It is easy to see that the Dirichlet form of $f_L$ is
bounded by $C_0 (N^2 \log L)^{-1}$ for some finite constant $C_0$.
Choosing $L=N^{1/2}$, we conclude that there exists a finite
constant $C_0$ such that
\begin{equation}
\label{65}
\Cap_N (\ms A_N, \ms B_N) \;\le\; \frac{C_0}{N^2 \, \log N} \;, \quad
d=2\;, \quad  \Cap_N (\ms A_N, \ms B_N) \;\le\; 
\frac{C_0}{N^d} \;, \quad d\ge 3\;.
\end{equation}
Condition \eqref{43} follows from this estimate and the definition of
the sequence $\theta_N$.

\smallskip\noindent\emph{Condition} ({\bf L1B}) in Lemma \ref{s26}.
By Lemma \ref{s22} and by the previous estimate of the capacity, there
exists a finite constant $C_0$ such that $\mf g_{\ms E} \le C_0 [N^2
\log N]^{-1}$ in dimension $2$ and $\mf g_{\ms E} \le C_0 N^{-d}$ in
dimension $d\ge 3$. Condition ({\bf L1B}) is thus fulfilled in view of
\eqref{70}.

\smallskip\noindent\emph{Condition} ({\bf L4}) in Theorem \ref{s00}.
We claim that there exists a sequence $T_N$ satisfying the conditions
({\bf L4}) if $\nu_N$ is a sequence of measures concentrated on $\ms
A_N$ and such that
\begin{equation}
\label{64}
\lim_{N\to\infty} \frac 1{R_N} \,
E_{\pi_{\ms E}} \Big[ \Big( \frac{\nu_N}{\pi_{\ms E}} \Big)^2 \Big]
\;=\; 0\;, 
\end{equation}
where $R_N = \log N$ in dimension $d=2$, and $R_N = N^{d-2}$ in
dimension $d\ge 3$. Let $M_N$ be an increasing sequence, $M_N \gg 1$,
for which \eqref{64} still holds if multiplied by $M_N$.  Since
$\Cap_{\star} (\ms A^{\star}_N , \ms B_N) \le \Cap_{\ms E} (\ms A_N ,
\ms B_N)$, by Corollary \ref{s19}, by \cite[Lemma 6.9]{bl2}, by
\eqref{65} and by \eqref{64} multiplied by $M_N$,
\begin{equation*}
\lim_{N\to\infty} \bb P^{\ms E}_{\nu_N} \big[ H_{\ms B_N} \le N^2 \,
M_N \big] \;=\; 0\;.
\end{equation*}

The strategy proposed in Section \ref{sec3} permits to weaken
assumption \eqref{64}.

\begin{lemma}
\label{s23}
Let $T_N$ be a sequence such that $T_N \ll \alpha_N^2 \, N^2 \,\log N$
in dimension $2$, and $T_N \ll \alpha_N^d \, N^d$ in dimension $d\ge
3$. Then,
\begin{equation*}
\lim_{N\to\infty} \, \max_{\eta \in \ms A_N}\, \bb P_{\eta} 
\big[ H_{Q_N} \le T_N \big] \;=\; 0\;.
\end{equation*}
\end{lemma}

\begin{proof}
In view of the definition of $\alpha_N$, we may assume that $T_N\gg
N^2$.  We present the arguments in dimension $2$, the case of higher
dimension being similar.  Fix a sequence $\eta^N \in \ms A_N$. Let
$\gamma_N = T^{-1}_N$ and denote by $\eta^\star(t)$ the
$\gamma$-enlargement of the process $\eta(t)$ on $V_N \cup V^\star_N$,
as defined in Section \ref{sec0}. Here, $V^\star_N$ represents a copy
of $V_N$, and the process $\eta^\star(t)$ jumps from $\eta$ to
$\eta^\star$ (and from $\eta^\star$ to $\eta$) at rate
$\gamma_N$. Denote by $\bb P^{\gamma}_{\eta}$ the probability measure
on the path space $D(\bb R_+, V_N \cup V^\star_N)$ induced by the
Markov process $\eta^\star(t)$ starting from $\eta$.

Let $W$ be the equilibrium potential $W (\eta) = \bb P^{\gamma}_{\eta}
[ H_{0} < H_{\breve{Q}^{\star}_N} ]$, where $0$ represents the
origin. In view of Lemma \ref{s11}, it is enough to show that
$W(\eta^N)$ vanishes as $N\uparrow\infty$. By the Dirichlet principle,
\begin{equation}
\label{67}
\< W , (-L_\star) W \>_{\pi_\star} 
\;=\;\Cap_{\star} (0 , \breve{Q}^{\star}_N) \;=\;
\inf_f \, \< f, (-L_\star) f\>_{\pi_\star} \;,
\end{equation}
where the infimum is carried over all functions $f$ which vanish on
$\breve{Q}^{\star}_N$ and which are equal to $1$ at the origin.  Using
the function $f_L$ introduced in \eqref{66}, we may show that the last
term is bounded by $C_0 (N^2 \log N)^{-1}$ for some finite constant
$C_0$. We used here the fact that $\gamma_N \ll (N \,\log N)^{-1}$.

Denote by $\prec$ the partial order of $\bb Z^d$ so that $\eta\prec
\xi$ if $\eta_j\le\xi_j$ for $1\le j\le d$.  A coupling argument shows
that the equilibrium potential $W$ is monotone on $\breve{Q}_N$: $W
(\eta) \le W(\xi)$ for $\eta\prec\xi$, $\eta$, $\xi\in \breve{Q}_N$.
Suppose that $W(\eta^N)$ does not vanish as $N\uparrow\infty$. In this
case there exists $\epsilon >0$ and a subsequence $N_j$, still denoted
by $N$, such that $W(\eta^N)\ge \epsilon$ for all $N$. Let $U_N =
\{\xi\in\breve{Q}_N : \eta^N \prec \xi\}$. By monotonicity of the
equilibrium potential, $W(\xi) \ge W(\eta^N)\ge \epsilon$ for all
$\xi\in U_N$. Therefore,
\begin{equation*}
\< W , (-L_\star) W \>_{\pi_\star} \;\ge\; \gamma_N \sum_{\xi\in U_N}
\pi_\star(\xi) W(\xi)^2 \;\ge\; c_0\, \gamma_N\, \epsilon^2
\,\alpha_N^2 
\end{equation*}
for some positive constant $c_0$. This contradicts the estimate
\eqref{67} because $\gamma_N \gg (\alpha^2_N\, N^2 \log N)^{-1}$.
\end{proof}

\smallskip\noindent\emph{Condition} ({\bf L4U}).  The proof of Lemma
\ref{s23} shows that condition ({\bf L4U}) is in force.

\begin{lemma}
\label{s27}
let $T_N$ be a sequence such that $T_N \ll \alpha^2_N \, N^2 \,\log N$
in dimension $2$, and $T_N \ll \alpha^d_N \, N^d$ in dimension $d\ge
3$. Then,
\begin{equation*}
\lim_{N\to\infty} \, \max_{\eta \in \ms A_N}\, \bb P^{\ms E}_{\eta} 
\big[ H_{\ms B_N} \le T_N \big] \;=\; 0\;. 
\end{equation*}
\end{lemma}

\begin{proof}
Consider the case of dimension $2$.  In view of the definition of
$\alpha_N$, we may assume that $T_N\gg N^2$. Let $\gamma_N =
T^{-1}_N$, and fix a sequence $\eta^N \in \ms A_N$. By the proof of
Lemma \ref{s23}, it is enough to show that $\bb P^{\star,
  \gamma}_{\eta^N} [ H_{\ms B_N} < H_{\ms A^\star_N} ]$ vanishes as
$N\uparrow\infty$, where $\bb P^{\star, \gamma}_{\eta}$ has been
introduced in Section \ref{sec0} just after the definition of
enlargements. Clearly,
\begin{equation*}
\bb P^{\star, \gamma}_{\eta^N} \big [ H_{\ms B_N} < H_{\ms A^\star_N}
\big ] \;=\; 
\bb P^{\gamma}_{\eta^N} \big [ H_{\ms B_N} < H_{\ms A^\star_N} \big ] 
\;\le\; 
\bb P^{\gamma}_{\eta^N} \big [ H_{0} < H_{\ms A^\star_N} \big ]\;, 
\end{equation*}
where $\bb P^{\gamma}_{\eta}$ is the probability measure introduced in
the proof of Lemma \ref{s23}. Denote by $\eta^{{\bf r},\breve{Q}} (t)$
the process $\eta(t)$ reflected at $\breve{Q}_N$ and by $\eta^{{\bf
    r}, \breve{Q}, \gamma} (t)$ the $\gamma$-enlargement of the
process $\eta^{{\bf r},\breve{Q}} (t)$ on $\breve{Q}_N \cup
\breve{Q}^\star_N$.  The last probability is clearly equal to $\bb
P^{{\bf r}, \breve{Q}, \gamma}_{\eta^N} [ H_{0} < H_{\ms A^\star_N}
]$, where $\bb P^{{\bf r}, \breve{Q}, \gamma}_{\eta}$ is the law of the
process $\eta^{{\bf r}, \breve{Q}, \gamma} (t)$ starting from $\eta$.

Let $\ms A^j_N = \{\eta\in \breve{Q}_N :
\eta_j > -\alpha_N\, N\}$, $j=1,2$, so that $\breve{Q}_N = \ms A_N \cup \ms
A^1_N \cup \ms A^2_N$ and
\begin{equation*}
\bb P^{{\bf r}, \breve{Q}, \gamma}_{\eta^N} 
\big [ H_{0} < H_{\ms A^\star_N} \big ] \;\le\;
\bb P^{{\bf r}, \breve{Q}, \gamma}_{\eta^N} 
\big [ H_{0} < H_{\breve{Q}^\star_N} \big ] \; +\;
\sum_{j=1}^2 \bb P^{{\bf r}, \breve{Q}, \gamma}_{\eta^N} 
\big [ H_{\ms A^{j,\star}_N} <  H_{\ms A^{\star}_N} \big ] \;.
\end{equation*}
We have shown in the proof of Lemma \ref{s23} that the first term on
the right hand side of the previous formula vanishes as
$N\uparrow\infty$. The other two are one-dimensional problems.

Let $W(\eta)$ be the equilibrium potential $\bb P^{{\bf r}, \breve{Q},
  \gamma}_{\eta} [ H_{\ms A^{1,\star}_N} < H_{\ms A^{\star}_N} ]$. We
claim that
\begin{equation*}
\lim_{N\to\infty} \max_{\eta\in \breve{Q}_N} \, W(\eta) \;=\; 0\;. 
\end{equation*}

Let $R_N$ be a sequence such that $N^2 \ll R_N \ll T_N$. With respect
to the measure $\bb P^{{\bf r}, \breve{Q}, \gamma}_{\eta}$,
$H_{\breve{Q}^\star_N}$ is a mean $T_N$ exponential time. Hence, $\bb
P^{{\bf r}, \breve{Q}, \gamma}_{\eta} [ H_{\breve{Q}^\star_N} < R_N]$
vanishes as $N\uparrow\infty$. It is therefore enough to show that
\begin{equation*}
  \lim_{N\to\infty} \bb P^{{\bf r}, \breve{Q}, \gamma}_{\eta} 
\big [ H_{\ms A^{1,\star}_N} < H_{\ms A^{\star}_N} \,,\, 
H_{\breve{Q}^\star_N} \ge R_N \big] \;=\; 0\;.
\end{equation*}
By the Markov property, the previous probability is equal to
\begin{equation*}
\bb E^{{\bf r}, \breve{Q}, \gamma}_{\eta} 
\Big [ \mb 1\{H_{\breve{Q}^\star_N} \ge R_N\} \,
\bb P^{{\bf r}, \breve{Q}, \gamma}_{\eta^{{\bf r},\breve{Q}}  (R_N)} 
\big [ H_{\ms A^{1,\star}_N} < H_{\ms  A^{\star}_N} \big] \, \Big] \;,
\end{equation*}
where $\eta^{{\bf r},\breve{Q}} (t)$ is the process $\eta(t)$
reflected at $\breve{Q}_N$. We bound last expectation by removing the
indicator of the set $H_{\breve{Q}^\star_N} \ge R_N$ and we estimate the
remaining term by
\begin{equation*}
\bb P^{{\bf r}, \breve{Q}, \gamma}_{\pi_{\breve{Q}}} \big [ H_{\ms A^{1,\star}_N} <
H_{\ms A^{\star}_N} \big] \;+\; \Vert \delta_\eta S^{{\bf
    r},\breve{Q}} (R_N) - \pi_{\breve{Q}} \Vert_{VT}\;,
\end{equation*}
where $\pi_{\breve{Q}}$ is the uniform measure on $\breve{Q}$ and
$S^{{\bf r},\breve{Q}} (t)$ the Markov semigroup of the process
$\eta^{{\bf r},\breve{Q}} (t)$. As $R_N \gg N^2$, which is the mixing
time of $\eta^{{\bf r},\breve{Q}} (t)$, the second term vanishes as
$N\uparrow\infty$, while the first term is the expectation of the
equilibrium potential $W$ with respect to the measure
$\pi_{\breve{Q}}$. If $L_{{\bf r},\breve{Q}}$ represents the generator
of the Markov process $\eta^{{\bf r},\breve{Q}} (t)$, we have that
$L_{{\bf r},\breve{Q}} W - \gamma W = - \gamma \mb 1\{\ms
A^{1}_N\}$. Taking the expectation with respect to $\pi_{\breve{Q}}$,
we conclude that $E_{\pi_{\breve{Q}}}[W] = \pi_{\breve{Q}} (\ms
A^{1}_N)$, which vanishes as $N\uparrow\infty$. This concludes the
proof of the lemma.
\end{proof}

In view of Lemma \ref{s26}, we have just shown that all assumptions of
Theorem \ref{s00} and Lemma \ref{s15} are in force. Moreover, by
\eqref{70} and Lemma \ref{s23}, the hypotheses of Lemma \ref{s25} are
fulfilled for $\ms D_N = \ms A_N$, $\ms F_N = \breve{Q}_N$ and $N^2
\ll T_N \ll \alpha^2_N \, (\log N) \, N^2$. Hence,

\begin{proposition}
\label{s24}
Consider the Markov process $\eta(t)$ on the dog graph. Assume that
the initial state $\nu_N$ is concentrated on $\ms A_N$. Then, the
time-rescaled order ${\mb X}^N_t = X^N_{\mf g_{\ms E}^{-1} t}$
converges to the Markov process on $\{1,2\}$ which starts from $1$ and
jumps from $x$ to $3-x$ at rate $1/2$. Moreover, in the time scale
$\mf g_{\ms E}^{-1}$ the time spent by the original process $\eta(t)$
on the set $\Delta_N = V_N \setminus \ms E_N$ is negligible.
\end{proposition}

As a last step, we replace in the previous statement the spectral gap
$\mf g_{\ms E}$ of the trace process by the spectral gap $\mf g$
of the original process.  Let $T_N$ be a sequence such that $N^2\ll
T_N \ll \alpha^2_N\, N^2 \,\log N$ in dimension $2$, and $N^2\ll T_N
\ll \alpha^d_N \, N^d$ in dimension $d\ge 3$. It follows from Lemma
\ref{s23} and from \eqref{70} that
\begin{equation*}
\begin{split}
& \lim_{N\to\infty} \, \min_{\eta \in \ms A_N}\, \bb P_{\eta} 
\big[ \eta (T_N) \in \breve{Q}_N \big] \;=\; 1\;,\\
&\quad 
\lim_{N\to\infty} \, \max_{\eta \in \ms A_N}\, \Vert \delta_{\eta} S (T_N) 
- \pi_{\breve{Q}_N} \Vert_{TV} \;=\; 0\;,
\end{split}
\end{equation*}
where $S(t)$ is the semigroup of the Markov process $\eta(t)$. These
estimates are the two ingredients needed in the proof of
\cite[Proposition 2.9]{clmst}, a result which states, among other
things, that there exists a mean zero eigenfunction $f_N$ of the
generator $L_N$ associated to the eigenvalue $\mf g$ such that
$\lim_N \Vert f_N \Vert_\infty = 1$. By this result and Proposition
\ref{s05}, $\lim_N (\mf g/\mf g_{\ms E})=1$.

\subsection{A polymer in the depinned phase \cite{cmt, clmst}.}
Fix $N\ge 1$ and denote by $E_N$ the set of all lattice paths starting
at $0$ and ending at $0$ after $2N$ steps:
\begin{equation*}
E_N\;=\; \{ \eta\in \bb Z^{2N+1}: \eta_{-N} = \eta_N = 0 \,,\,
\eta_{j+1}-\eta_j = \pm 1 \,,\, -N \le j < N \}\;.
\end{equation*}
Fix $0<\alpha<1$ and consider the dynamics on $E_N$ induced by the
generator $L_N$ defined by
\begin{equation*}
(L_N f)(\eta) \;=\; \sum_{j=-N+1}^{N-1} c_{j,+} (\eta) [f(\eta^{j,+})
- f(\eta)] \;+\; \sum_{j=-N+1}^{N-1} c_{j,-} (\eta) [f(\eta^{j,-})
- f(\eta)]\;,
\end{equation*}
for every function $f: E_N\to\bb R$. In this formula $\eta^{j,\pm}$
represents the configuration which is equal to $\eta$ at every site
$k\not = j$ and which is equal to $\eta_j \pm 2$ at site $j$. The jump
rate $c_{j,+} (\eta)$ vanishes at configurations $\eta$ which do not
satisfy the condition $\eta_{j-1} = \eta_{j+1} = \eta_j +1$, and 
it is given by
\begin{equation*}
c_{j,+} (\eta) \;=\; 
\begin{cases}
1/2 & \text{if $\eta_{j-1} = \eta_{j+1} \not = \pm 
  1$,} \\
1/[(1+\alpha)] & \text{if $\eta_{j-1} = \eta_{j+1} = 
  1$,} \\
\alpha/[(1+\alpha)] & \text{if $\eta_{j-1} = \eta_{j+1} = -
  1$} 
\end{cases}
\end{equation*}
for configuration which fulfill the condition $\eta_{j-1} = \eta_{j+1}
= \eta_j +1$. Let $-\eta$ stand for the configuration $\eta$ reflected
at the origin, $(-\eta)_j = - \eta_j$, $-N\le j\le N$. The rates
$c_{j,-}(\eta)$ are given by $c_{j,-}(\eta) = c_{j,+}(-\eta)$.

Denote by $\Sigma(\eta)$ the number of zeros in the path $\eta$,
$\Sigma (\eta) = \sum_{-N\le j\le N} \mb 1\{\eta_j=0\}$. The
probability measure $\pi_N$ on $E_N$ defined by $\pi_N(\eta) =
Z^{-1}_{2N} \alpha^{\Sigma(\eta)}$, where $Z_{2N}$ is a normalizing
constant, is easily seen to be reversible for the dynamics generated
by $L_N$.

By \cite[Theorem 3.5]{cmt}, the spectral gap $\mf g$ is bounded
above by $C(\alpha) (\log N)^8 / N^{5/2}$ for some finite constant
$C(\alpha)$. Following \cite{clmst}, let $\ms E^1_N$ be the set of
configurations in $E_N$ such that $\eta_j>0$ for all $-(N - \ell) < j
< (N-\ell)$, where $\ell=\ell_N$ is a sequence such that $1\ll \ell_N
\ll N$, and let $\ms E^2_N = \{\eta\in E_N : -\eta \in \ms E^1_N\}$,
$\Delta_N = E_N \setminus (\ms E^1_N \cup \ms E^2_N)$. By equation
(2.27) in \cite{clmst}, $\pi(\ms E^1_N) = \pi(\ms E^1_N) = (1/2) +
O(\ell^{-1/2})$. Moreover, taking $\ell_N = (\log N)^{1/4}$, by
\cite[Proposition 2.6]{clmst}, for every $\epsilon >0$, there exists
$N_0$ such that for all $N\ge N_0$, $\mf g_{{\mb r}, 1} =
\mf g_{{\mb r}, 2} \ge N^{-(2+\epsilon)}$. In conclusion, choosing
$\epsilon$ small enough and $\ell_N = (\log N)^{1/4}$,
\begin{equation*}
\mf g \;\ll\; \min\big\{\mf g_{{\mb r}, 1} \,,\,
\mf g_{{\mb r}, 2} \big\}
\end{equation*}
for all $N$ large enough, which proves that condition ({\bf L1B}) is
in force.

By \cite[Proposition 2.9]{clmst}, there exists an eigenfunction $f$ of
the generator $L_N$ such that $E_{\pi}[f]=0$, $E_{\pi}[f^2]=1$, $L_N f
= \mf g f$ and $\Vert f\Vert_{\infty} = 1 + o_N(1)$ where $\Vert
f\Vert_{\infty}$ represents the sup norm of $f$ and $o_N(1)$ an
expression which vanishes as $N\uparrow\infty$. Therefore, since
$\pi(\Delta_N) \to 0$, by Proposition \ref{s05}, $\mf g/\mf g_{\ms E}$
converges to $1$ as $N\uparrow\infty$. 

Let $\nu_N$ be a sequence of probability measures concentrated on $\ms
E^1_N$ and satisfying condition \eqref{39}. For example, one may
define $\nu_N (\,\cdot\,)$ as $\pi(\,\cdot\,|\, \ms F\,)$, where $\ms
F$ is a subset of $\ms E^1_N$ such that $\liminf_{N\to\infty} \pi(\ms
F) \ge c_0$ for some positive constant $c_0$.  Define the trace
process $\eta^{\ms E} (t)$ and the order $X^N_t$ as in Section
\ref{sec0}.  By Proposition \ref{s05} and Lemma \ref{s26}, and in view
of the previous remarks, the time-rescaled process ${\mb X}^N_t =
X^N_{t/\mf g}$ converges to a Markov process on $\{1,2\}$ which starts
from $1$ and jumps from $x$ to $3-x$ at rate $1/2$. Moreover, by Lemma
\ref{s09}, the time spent by the process $\eta(t)$ on the time scale
$\mf g^{-1}$ outside the set $\ms E_N$ is negligible.

The difference between this result, derived from a general statement,
and Theorems 1.3 and 1.5 in \cite{clmst} is that we require in Theorem
\ref{s02} the initial state to be close to the stationary state of the
reflected process in one of the wells, while \cite{clmst} allows the
process to start from any state in one of the wells. This strong
assumption on the initial condition permits to consider larger wells
and to have an explicit description of these wells. To prove tunneling
for a process starting from a state, one needs to show that the mixing
conditions ({\bf L4U}) are in force.

\medskip\noindent{\bf Acknowledgments}. The authors wish to thank
A. Gaudilli\`ere, M. Jara, H. Lacoin, M. Loulakis and A. Teixeira for
stimulating discussions.

\end{document}